\documentclass{amsart}

\usepackage[all, cmtip]{xy}
\usepackage{mathtools}
\usepackage[T1]{fontenc}

\newtheorem{theorem}{Theorem}[section]
\newtheorem{lemma}[theorem]{Lemma}
\newtheorem{proposition}[theorem]{Proposition}

\theoremstyle{definition}
\newtheorem{definition}[theorem]{Definition}
\newtheorem{example}[theorem]{Example}

\newtheorem{assumption}[theorem]{Assumption}

\theoremstyle{remark}
\newtheorem{remark}[theorem]{Remark}

\newcommand{\Q}{\mathbb{Q}}

\newcommand{\R}{\mathbb{R}}
\newcommand{\LL}{\mathbb{L}}
\newcommand{\Proj}{\mathbb{P}}
\newcommand{\WP}{\mathbb{W}\mathbb{P}}
\newcommand{\calO}{\mathcal{O}}

\newcommand{\calE}{\mathcal{E}}

\newcommand{\calF}{\mathcal{F}}

\newcommand{\calM}{\mathcal{M}}
\newcommand{\calN}{\mathcal{N}}
\newcommand{\calA}{\mathcal{A}}
\newcommand{\calX}{\mathcal{X}}
\newcommand{\calY}{\mathcal{Y}}
\DeclareMathOperator{\NS}{NS}
\DeclareMathOperator{\Kum}{Kum}
\DeclareMathOperator{\Pic}{Pic}
\DeclareMathOperator{\rank}{rank}

\numberwithin{table}{section}
\numberwithin{equation}{section}

\begin{document}

\title[Calabi-Yau threefolds fibred by Kummer surfaces]{Calabi-Yau threefolds fibred by Kummer surfaces associated to products of elliptic curves}

\author[C. F. Doran]{Charles F. Doran}
\address{Department of Mathematical and Statistical Sciences, 632 CAB, University of Alberta, Edmonton, AB, T6G 2G1, Canada}
\email{charles.doran@ualberta.ca}
\thanks{C. F. Doran and A. Y. Novoseltsev were supported by the Natural Sciences and Engineering Resource Council of Canada (NSERC), the Pacific Institute for the Mathematical Sciences (PIMS), and the McCalla Professorship at the University of Alberta}

\author[A. Harder]{Andrew Harder}
\address{Department of Mathematical and Statistical Sciences, 632 CAB, University of Alberta, Edmonton, AB, T6G 2G1, Canada}
\email{aharder@ualberta.ca}
\thanks{A. Harder was supported by an NSERC Post-Graduate Scholarship}

\author[A.Y. Novoseltsev]{Andrey Y. Novoseltsev}
\address{Department of Mathematical and Statistical Sciences, 632 CAB, University of Alberta, Edmonton, AB, T6G 2G1, Canada}
\email{novoselt@ualberta.ca}

\author[A. Thompson]{Alan Thompson}
\address{Department of Pure Mathematics, University of Waterloo, 200 University Ave West, Waterloo, ON, N2L 3G1, Canada}
\email{am6thomp@uwaterloo.ca}
\thanks{A. Thompson was supported by a Fields-Ontario-PIMS Postdoctoral Fellowship with funding provided by NSERC, the Ontario Ministry of Training, Colleges and Universities, and an Alberta Advanced Education and Technology Grant}


\date{}

\begin{abstract} We study threefolds fibred by Kummer surfaces associated to products of elliptic curves, that arise as resolved quotients of threefolds fibred by certain lattice polarized K3 surfaces under a fibrewise Nikulin involution. We present a general construction for such surfaces, before specializing our results to study Calabi-Yau threefolds arising as resolved quotients of threefolds fibred by mirror quartic K3 surfaces. Finally, we give some geometric properties of the Calabi-Yau threefolds that we have constructed, including expressions for Hodge numbers.
\end{abstract}

\maketitle

\section{Introduction}

Building on earlier work by Shioda, Inose \cite{sk3s}\cite{desk3sni}, Nikulin \cite{isbfa} and Morrison \cite{k3slpn}, Clingher and Doran  \cite{milpk3s}\cite{ngik3s} exhibited a duality between K3 surfaces admitting a lattice polarization by the lattice
\[M := H \oplus E_8 \oplus E_8\]
and Kummer surfaces associated to products of elliptic curves, that closely relates the geometry of the surfaces on each side. This duality is easy to describe: any $M$-polarized K3 surface admits a canonically defined Nikulin involution, the resolved quotient by which is a Kummer surface associated to a product of elliptic curves and, conversely, a Kummer surface associated to a product of elliptic curves also admits a Nikulin involution, the resolved quotient by which is isomorphic to an $M$-polarized K3 surface. Moreover, applying this process twice returns us to the surface we started with.

This duality was exploited in \cite{14thcase}, to obtain certain geometric properties of a Calabi-Yau threefold admitting a fibration by $M$-polarized K3 surfaces. In that case, it was proven that Clingher's and Doran's construction could be performed fibrewise, giving rise to a new Calabi-Yau threefold that was fibred by Kummer surfaces associated to products of elliptic curves. As it turns out, the geometry of the Kummer fibred threefold thus obtained was easier to study, and could be used to derive geometric properties of the original threefold fibred by $M$-polarized K3 surfaces.

The main aim of this paper is to investigate to what extent this construction can be generalized to arbitrary threefolds fibred by $M$-polarized K3 surfaces.

More precisely, suppose that $\calX$ is a threefold fibred by K3 surfaces and that the restriction $\calX_U$ of this fibration to the locus of smooth fibres is an $M$-polarized family of K3 surfaces, in the sense of \cite[Definition 2.1]{flpk3sm}. Then results of \cite{flpk3sm} show that the canonical Nikulin involution on the fibres of $\calX_U$ extends to the entire threefold, so we may quotient and resolve singularities to obtain a threefold fibred by Kummer surfaces $\calY_U$. We may then ask whether $\calY_U$ can be compactified to a threefold $\calY$ and, if so, what properties this new threefold has.

One case that is of particular interest is when the threefold $\calX$ is Calabi-Yau, as occurred in \cite{14thcase}. In this case, one would like to know whether $\calY$ (if it exists) is also Calabi-Yau and, if so, how its properties relate to those of $\calX$.

In the latter part of this paper we address this second question in a special case, where the Calabi-Yau threefolds $\calX$ are very well-understood. Specifically, we consider the setting where $\calX$ is one of the Calabi-Yau threefolds $\calX_g$ fibred by mirror quartic K3 surfaces constructed in \cite{cytfmqk3s}. Note that this is a special case of the construction above as, by definition, a mirror quartic K3 surface is polarized by the lattice $M_2 := H \oplus E_8 \oplus E_8 \oplus \langle -4\rangle$, which clearly contains $M$ as a primitive sublattice. These threefolds $\calX_g$ encompass many well-known examples, including the quintic mirror threefold, and provide a useful illustration of our methods in a concrete setting.

In this special case, we show that we can explicitly construct Kummer surface fibred threefolds $\calY_g$ that are related to the $\calX_g$ by a fibrewise quotient-resolution procedure as above. Moreover, the $\calY_g$ are Calabi-Yau in most cases and have geometric properties that are closely related to those of the $\calX_g$. This gives a new perspective from which to study the geometry of the threefolds $\calX_g$, amongst them the quintic mirror.

Finally, we note that the construction of the threefolds $\calY_g$ is somewhat interesting in its own right, as they are all constructed from a single, rigid Calabi-Yau threefold. This rigid Calabi-Yau threefold is in turn built from a well-known extremal rational elliptic surface, using a method originally due to Schoen \cite{ofpress} that was later extended by Kapustka and Kapustka \cite{fpessakf}.
\medskip

The structure of this paper is as follows. In Section \ref{backgroundsect} we review some background material,  mostly taken from \cite{milpk3s} and \cite{ngik3s}, about $M$-polarized K3 surfaces, and describe the threefolds $\calX$ that are fibred by them. Then, in Section \ref{sect:threefoldkummer}, we develop the theory required to construct the associated threefolds fibred by Kummer surfaces $\calY$, and describe their construction in general terms. This construction proceeds by first \emph{undoing the Kummer construction}, as originally described in \cite[Section 4.3]{flpk3sm}, then running a generalized version of the \emph{forward construction} from \cite[Section 7]{14thcase}.

Finally, in Section \ref{CYsect}, we specialize the entire discussion to the case where $\calX$ is one of the Calabi-Yau threefolds fibred by quartic mirror K3 surfaces $\calX_g$ constructed in \cite{cytfmqk3s}. In this case we can construct the associated threefolds fibred by Kummer surfaces $\calY_g$ completely explicitly as pull-backs of a special threefold $\calY_2$. This special threefold is constructed in turn as a resolved quotient of a rigid Calabi-Yau threefold $\calA_2$, which is described in Section \ref{section:undokummer}. The properties of $\calA_2$ are then studied in Section \ref{A2sect}, after which we carefully describe the quotient-resolution procedure used to obtain $\calY_2$ from it in Sections \ref{sect:D8action} and \ref{singfibsect}. The method for constructing the $\calY_g$ from $\calY_2$ is detailed in Section \ref{cy3constsect}, and some of the properties of the $\calY_g$ are computed in Section \ref{propertiessect}.

\section{Background material} \label{backgroundsect}

We begin by setting up some notation. Let $\calX$ be a smooth projective threefold that admits a fibration $\calX \to B$ by K3 surfaces over a smooth curve. Let $N := \NS(X_p)$ denote the N\'{e}ron-Severi group of the fibre of $\calX$ over a general point $p \in B$. Suppose that there exists a primitive lattice embedding $M \hookrightarrow N$ of the lattice $M:= H \oplus E_8 \oplus E_8$ into $N$ (we will assume that such an embedding has been fixed in what follows).

Denote the open set over which the fibres of $\calX$ are smooth K3 surfaces by $U \subset B$ and let $\calX_U \to U$ denote the restriction of $\calX$ to $U$. Suppose further that $\calX_U \to U$ is an $N$-polarized family of K3 surfaces, in the sense of \cite[Definition 2.1]{flpk3sm}.

\begin{remark} Nineteen such fibrations are known on Calabi-Yau threefolds $\calX$ with $h^{2,1}(\calX) = 1$; these are summarized by \cite[Table 5.1]{flpk3sm}. Moreover, a large class of additional examples of Calabi-Yau threefolds fibred by K3 surfaces polarized by the lattice $M_2 := H \oplus E_8 \oplus E_8 \oplus \langle -4 \rangle$ are constructed in \cite{cytfmqk3s}; we will return to these examples in Section \ref{CYsect}.
\end{remark}

\subsection{$M$-polarized K3 surfaces} \label{sect:Mpol}

By assumption, a general fibre $X_p$ of $\calX$ is an $M$-polarized K3 surface. We recall here some basic properties of $M$-polarized K3 surfaces, that will be used repeatedly in what follows. In this section we will denote an $M$-polarized K3 surface by $(X,i)$, where $X$ is a K3 surface and $i$ is a primitive lattice embedding $i \colon M \hookrightarrow \mathrm{NS}(X)$.

Building upon work of Inose \cite{desk3sni}, Clingher, Doran, Lewis and Whitcher \cite{nfk3smmp} have shown that $M$-polarized K3 surfaces have a coarse moduli space given by the locus $d \neq 0$ in the weighted projective space $\WP(2,3,6)$ with weighted coordinates $(a,b,d)$. Thus, by normalizing $d = 1$, we may associate a pair of complex numbers $(a,b)$ to an $M$-polarized K3 surface $(X,i)$.

The first piece of structure that we need on $(X,i)$ comes from the work of Morrison \cite{k3slpn}, who showed that the composition of $i$ with the canonical embedding $E_8 \oplus E_8 \hookrightarrow M$ defines a canonical \emph{Shioda-Inose structure} on $(X,i)$ (named for Shioda and Inose \cite{sk3s}, who were the first to study such structures). By definition, such a structure consists of a Nikulin involution $\beta$ on $X$, such that the resolved quotient $Y = \widetilde{X/\beta}$ is a Kummer surface and there is a Hodge isometry $ T_Y \cong T_X(2)$, where $T_X$ and $T_Y$ denote the transcendental lattices of $X$ and $Y$ respectively, and $T_X(2)$ indicates that the bilinear pairing on $T_X$ has been multiplied by $2$.

By \cite[Theorem 3.13]{milpk3s}\footnotemark[1], we see that in our setting $Y$ is isomorphic to the Kummer surface $\Kum(A)$, where $A \cong E_1 \times E_2$ is an Abelian surface that splits as a product of elliptic curves. By \cite[Corollary 4.2]{milpk3s}\footnotemark[1] the $j$-invariants of these elliptic curves are given by the roots of the equation
\[ j^2 - \sigma j + \pi = 0,\]
where $\sigma$ and $\pi$ are given in terms of the $(a,b)$ values associated to $(X,i)$ by $\sigma = a^3 - b^2+1$ and $\pi = a^3$. Label the exceptional $(-2)$-curves in $Y$ arising from the resolution of the singularities of $X/\beta$ by $\{F_1,\ldots,F_8\}$.
\footnotetext[1]{We note that equivalent results to those attributed to Clingher and Doran \cite{milpk3s} here were proved independently by Shioda \cite{kstcek3s}, using a slightly different characterization of $M$-polarized K3 surfaces as elliptically fibred K3 surfaces with section and two fibres of type $II^*$.}

There is one more piece of structure on $(X,i)$ that we will need in our discussion. By \cite[Proposition 3.10]{milpk3s}, the K3 surface $X$ admits two uniquely defined elliptic fibrations $\Theta_{1,2}\colon X \to \Proj^1$, the \emph{standard} and \emph{alternate fibrations}. We will be mainly concerned with the alternate fibration $\Theta_2$. This fibration has two sections, one singular fibre of type $I_{12}^*$ and, if $a^3 \neq (b \pm 1)^2$, six singular fibres of type $I_1$ \cite[Proposition 4.6]{milpk3s}. 

The alternate fibration $\Theta_2$ is preserved by the Nikulin involution $\beta$, so induces an elliptic fibration $\Psi\colon Y \to \Proj^1$ on the Kummer surface $Y$. The two sections of $\Theta_2$ are identified to give a section $S$ of $\Psi$, and $\Psi$ has one singular fibre of type $I_{6}^*$ and, if $a^3 \neq (b \pm 1)^2$, six $I_2$'s \cite[Proposition 4.7]{milpk3s}. 

\begin{remark}\label{Nikulinrem1} As noted in the introduction, this construction is completely reversible. Clingher and Doran \cite[Section 1]{ngik3s} identify a second distinguished section $S'$ of the fibration $\Psi$, along with a uniquely defined Nikulin involution $\beta'$ on $Y$ that preserves $\Psi$ and takes $S$ to $S'$. The resolved quotient $\widetilde{Y / \beta'}$ is then isomorphic to $X$, and $\Psi$ induces the alternate fibration $\Theta_2$ on $X$. \end{remark}

The locations of the $I_2$ fibres in $\Psi$ are given by \cite[Proposition 4.7]{milpk3s}. They occur at the roots of the polynomials $(P(x) \pm 1)$, where $P$ is the cubic equation
\begin{equation} \label{Pequation} P(x) := 4x^3 - 3ax - b, \end{equation}
$(a,b)$ are the modular parameters associated to $X$, and $x$ is an affine coordinate on $\Proj^1$ chosen so that the $I_6^*$ fibre occurs at $x = \infty$.

Finally, using this information we may identify some of the $(-2)$-curves $F_i$ in $Y$. By the discussion in \cite[Section 4.3]{flpk3sm}, $\{F_3,F_4,F_5\}$ (resp.$\{F_6,F_7,F_8\}$) are the $(-2)$-curves in the $I_2$ fibres lying over the roots of $(P(x)-1)$ (resp. $(P(x)+1)$) that do not meet the section $S$ (this labelling may seem arbitrary, but in fact is chosen to match with that used in \cite[Section 4.3]{flpk3sm}).

\section{Threefolds fibred by Kummer surfaces} \label{sect:threefoldkummer}

We will now apply this theory to study the K3-fibred threefold $\calX \to B$. Via the embedding $M \hookrightarrow N$, we see that a general fibre of $\calX$ is an $M$-polarized K3 surface. Thus, by the discussion in Section \ref{sect:Mpol}, there is a canonical Shioda-Inose structure on such a fibre, which defines a Nikulin involution on it. This involution extends uniquely to all fibres of $\calX_U$ by \cite[Corollary 2.12]{flpk3sm}. 

Let $\calY_U \to U$ denote the family obtained by taking the quotient of $\calX_U$ by this involution and resolving the resulting singularities. The discussion from Section \ref{sect:Mpol} shows that the fibres of $\calY_U$ are Kummer surfaces $\Kum(E_1\times E_2)$ associated to products of elliptic curves $E_1 \times E_2$. Furthermore, the alternate fibration $\Theta_2$ on the fibres of $\calX_U$ induces a uniquely defined elliptic fibration $\Psi$ on the fibres of $\calY_U$.

\begin{remark} \label{Nikulinrem2} As in the K3 surface case, this construction turns out to be reversible. Let $\beta'$ be the Nikulin involution on a general fibre of $\calY_U$, as described in Remark \ref{Nikulinrem1}. By the description of the action of monodromy in $U$ from \cite[Section 4.3]{flpk3sm} and the description of $\beta'$ from \cite[Section 1]{ngik3s}, it can be shown that the action of $\beta'$ and the action of monodromy on the N\'{e}ron-Severi lattice of a general fibre commute. So, by \cite[Proposition 2.11]{flpk3sm}, $\beta'$ extends to a involution on $\calY_U$, the resolved quotient by which is isomorphic to $\calX_U$. \end{remark}

Our aim is to explicitly construct K3-fibred threefolds $\calY$ over $B$, so that the restriction of $\calY$ to the open set $U \subset B$ is isomorphic to $\calY_U$, and to study their properties.

\subsection{Undoing the Kummer construction}

In order to do this, the first step is to \emph{undo the Kummer construction} for $\calY_U$, i.e. to find a family of Abelian surfaces $\calA_U \to U$ which gives rise to $\calY_U$ upon fibrewise application of the Kummer construction. To do this, we will use the results from \cite[Section 4.3]{flpk3sm}. However, in order to apply these results we need to make the following assumption \cite[Assumption 4.6]{flpk3sm}; unless otherwise stated, we will make this assumption throughout the remainder of this section.

\begin{assumption} \label{I2ass} The fibration $\Psi$ on a general fibre $Y_p$ of $\calY_U$ has six singular fibres of type $I_2$.
\end{assumption}

\begin{remark} \label{I2rem} Note that each $I_2$ fibre in the fibration $\Psi$ on $Y_p$ arises as the total transform of an $I_1$ fibre in the alternate fibration $\Theta_2$ on $X_p$. Thus to check that Assumption \ref{I2ass} is satisfied, it is equivalent to show that the alternate fibration $\Theta_2$ on a general fibre $X_p$ of $\calX_U$ has six singular fibres of type $I_1$

This latter condition is easy to check numerically from the $(a,b)$ parameters associated to $X_p$. Indeed, the locations of the $I_1$ fibres in $\Theta_2$ are given by the roots of the polynomials $(P(x) \pm 1)$, where $P(x)$ is defined by Equation \eqref{Pequation}, which are all distinct if and only if $a^3 \neq (b \pm 1)^2$.
\end{remark}

Unfortunately, by the discussion in \cite[Section 4.3]{flpk3sm}, it is not always possible to undo the Kummer construction on $\calY_U$ directly. Instead, we must pull everything back to a cover $f\colon C \to B$.

This cover is constructed by the method described in \cite[Section 4.3]{flpk3sm}. Let $p \in U$ be a point and consider the six divisors $\{F_3,\ldots,F_8\}$ in the fibre $Y_p$ of $\calY$ over $p$. Monodromy in $U$ preserves the fibration $\Psi$ along with its section $S$ (as both are induced from the structure of the alternate fibration $\Theta_2$ on the fibres of $\calX_U$), so must act to permute the $F_i$. We thus have a homomorphism $\rho\colon \pi_1(U,p) \to S_6$; call its image $G$. Then define an unramified $|G|$-fold cover $f \colon V \to U$ as follows: the preimages of $p \in U$ are labelled by permutations in $G$ and, if $\gamma$ is a loop in $U$, monodromy around $f^{-1}(\gamma)$ acts on these labels as composition with $\rho(\gamma)$. This cover extends uniquely to a cover $f\colon C \to B$, with ramification over the points in $B-U$.

Let  $\calY'_V$ denote the pull-back of $\calY_U$ to $V$. Then \cite[Theorem 4.11]{flpk3sm} shows that we can undo the Kummer construction for $\calY'_V$, so there exists a family of Abelian surfaces $\calA_V \to V$ that gives rise to $\calY'_V$ under fibrewise application of the Kummer construction.

We have the following diagram:
\[\xymatrixcolsep{3pc}\xymatrix{ \calA_V \ar@{-->}[r]^{\mathrm{Kummer}} \ar[d] & \calY'_V \ar[r] \ar[d] & \calY_U \ar[d] & \calX_U \ar[d] \ar@{-->}[l]_{\mathrm{Nikulin}} \ar@{^{(}->}[r] & \calX \ar[d] \\
V \ar@{=}[r] & V \ar[r]^f & U \ar@{=}[r] & U \ar@{^{(}->}[r] & B
}\]

\subsection{The forward construction} \label{forwardconst}

Our next aim is to construct threefolds $\calA$, $\calY'$ and $\calY$ that agree with $\calA_V$, $\calY'_V$ and $\calY_U$ over $V$ and $U$ respectively. This construction will generalize the \emph{forward  construction} of \cite[Section 7]{14thcase}.

We begin by constructing a threefold fibred by Abelian surfaces $\calA \to C$ that agrees with $\calA_V$ over $V$. The first step is to identify some special divisors on the fibres of $\calY'_V$.

Recall that a fibre of $\calY'_V$ is isomorphic to $\mathrm{Kum}(E_1 \times E_2)$, where $E_1$ and $E_2$ are elliptic curves. There is a special configuration of twenty-four $(-2)$-curves on $\mathrm{Kum}(E_1 \times E_2)$ arising from the Kummer construction, that we shall now describe (here we note that we use the same notation as \cite[Definition 3.18]{milpk3s}, but with the roles of $G_i$ and $H_j$ reversed).

Let $\{x_0,x_1,x_2,x_3\}$ and $\{y_0,y_1,y_2,y_3\}$ denote the two sets of points of order two on $E_1$ and $E_2$ respectively. Denote by $G_i$ and $H_j$ ($0 \leq i,j \leq 3$) the $(-2)$-curves on $\mathrm{Kum}(E_1 \times E_2)$ obtained as the proper transforms of $E_1 \times \{y_i\}$ and $\{x_j\} \times E_2$ respectively. Let $E_{ij}$ be the exceptional $(-2)$-curve on $\mathrm{Kum}(E_1 \times E_2)$ associated to the point $(x_j,y_i)$ of $E_1 \times E_2$. This gives $24$ curves, which have the following intersection numbers:

\begin{align*}
G_i.H_j &= 0, \\
G_k.E_{ij} &= \delta_{ik}, \\
H_k.E_{ij} &= \delta_{jk}.
\end{align*}

\begin{definition} The configuration of twenty-four $(-2)$-curves 
\[ \{G_i,H_j,E_{ij} \mid 0 \leq i,j \leq 3 \}\]
is called a \emph{double Kummer pencil} on $\mathrm{Kum}(E_1 \times E_2)$.
\end{definition}

With this in place, we can prove an analogue of \cite[Lemma 7.4 and Proposition 7.5]{14thcase}. 

\begin{proposition} \label{ellsurfprop} $\calA_V \to V$ is isomorphic over $V$ to a fibre product $\calE_1 \times_C \calE_2$ of minimal elliptic surfaces $\calE_i \to C$ with section. Furthermore, the $j$-invariants of the elliptic curves $E_1$ and $E_2$ forming the fibres of $\calE_1$ and $\calE_2$ over a point $p \in C$ are given by the roots of the equation
\begin{equation}\label{jequation} j^2 - \sigma(p) j + \pi(p) = 0, \end{equation}
where $\sigma(p)$ and $\pi(p)$ are the $\sigma$ and $\pi$ invariants associated to the $M$-polarized K3 surface fibre $X_{f(p)}$ of $\calX_U$ over $f(p)$. 
\end{proposition}

\begin{remark} \label{distinctrem}  Thus, using the expressions for $(\sigma,\pi)$ in terms of $(a,b)$ from Section \ref{sect:Mpol}, we find that the discriminant of Equation \eqref{jequation} is
\[\sigma^2 - 4\pi = a^6 - 2a^3b - 2a^3 + b^4 - 2b^2 + 1 = (a^3-(b-1)^2)(a^3-(b+1)^2).\]
But Assumption \ref{I2ass} and Remark \ref{I2rem} imply that, for generic $p \in C$, this does not vanish, so the roots of Equation \eqref{jequation} are generically distinct.
\end{remark}

\begin{proof}[Proof of Proposition \ref{ellsurfprop}] We begin by showing that the fibration $\calA_V \to V$ has a section $s$. Construct a double Kummer pencil $\{G_i,H_j,E_{ij}\}$ on the fibre $Y'_p$ of $\calY'_V$ over $p \in V$ as described in \cite[Section 4.3]{flpk3sm}. By  \cite[Theorem 4.11]{flpk3sm}, $\calY'_{V}$ is $\NS(Y'_p)$-polarized, so the divisors in this pencil are invariant under monodromy around loops in $V$. In particular, the curve $E_{11}$ is invariant. So $E_{11}$ sweeps out a divisor in $\calY'_V$, which intersects each smooth fibre in a $(-2)$-curve. Passage to $\calA_V$ contracts this divisor to a curve which intersects each smooth fibre in a single point, i.e. a section over $V$.

Now, let $\calA_p$ denote the fibre of $\calA_V$ over $p$, which is isomorphic to a product $E_1 \times E_2$ of elliptic curves. We may identify $E_1$ (resp. $E_2$) with the preimages of the curve $G_1$ (resp. $H_1$) in the double Kummer pencil on $Y'_p$. As $G_1$ and $H_1$ are invariant under monodromy around loops in $V$, they sweep out two divisors in $Y'_V$. Upon passage to $\calA_V$ these two divisors become a pair of elliptic surfaces, $\calE_{1,V} \to V$ and $\calE_{2,V} \to V$, which intersect along the section $s$. By \cite[Theorem 2.5]{nak}, there are unique extensions of $\calE_{i,V} \to V$ to minimal elliptic surfaces $\calE_i \to C$ over $C$, for $i = 1,2$. By construction, we have an isomorphism over $V$ between $\calA_V$ and $\calE_1 \times_C \calE_2$, as required.

Finally, the statement about the $j$-invariants is an easy consequence of the discussion in Section \ref{sect:Mpol}. \end{proof}

Using this, we may construct a threefold $\calY' \to C$ that is isomorphic to $\calY'_V$ over $V$ by applying the Kummer construction to $\calE_1 \times_C \calE_2$ fibrewise. To further construct a model for $\calY_U$, we need to know how the group $G$ defining the cover $f$ acts on $\calE_1 \times_C \calE_2$.

\begin{lemma} \label{Gperm} Let $\varphi$ denote the action of a permutation in $G \subset S_6$ on $C$. Then either $\varphi$ induces automorphisms on $\calE_1$ and $\calE_2$, or $\varphi$ induces an isomorphism $\calE_1 \to \calE_2$.\end{lemma}
\begin{proof} Note first that $\varphi$ induces an automorphism $\hat{\varphi}$ of $\calA_V$. Furthermore, $\hat{\varphi}$ preserves the section $s$ as, by \cite[Lemma 4.5]{flpk3sm}, the curve $E_{11}$ in a general fibre $Y_p$ of $\calY_U$ is invariant under monodromy in $U$.

As in the proof of Proposition \ref{ellsurfprop}, we identify $\calE_{1,V}$ and $\calE_{2,V}$ with the elliptic surfaces in $\calA_V$ swept out by the preimages of the curves $G_1$ and $H_1$. These elliptic surfaces intersect along the section $s$. 

As $\hat{\varphi}$ preserves $s$, we see that $\hat{\varphi}(\calE_{1,V})$ is an elliptic surface in $\calA_V$ that contains $s$ as a section. It must therefore either be $\calE_{1,V}$ or $\calE_{2,V}$. The same holds for $\calE_{2,V}$. Thus, we see that $\varphi$ either induces automorphisms on $\calE_{1,V}$ and $\calE_{2,V}$, or induces an isomorphism $\calE_{1,V} \to \calE_{2,V}$.

Thus $\varphi$ induces either a birational automorphism on $\calE_1$ and $\calE_2$, or a birational map $\calE_{1} \to \calE_{2}$. But, by \cite[Proposition II.1.2]{btes}, a birational map between minimal elliptic surfaces is an isomorphism.\end{proof}

It is easy to determine which case of Lemma \ref{Gperm} occurs:

\begin{lemma} \label{lemma:jtracking} Let $\varphi$ denote the action of a permutation in $G \subset S_6$ on $C$. Then $\varphi$ induces automorphisms on $\calE_1$ and $\calE_2$ \textup{(}resp. $\varphi$ induces an isomorphism $\calE_1 \to \calE_2$\textup{)} if and only if the action of $\varphi$ preserves \textup{(}resp. exchanges\textup{)} the roots of Equation \eqref{jequation} \textup{(}which are generically distinct by Remark \ref{distinctrem}\textup{)}.
\end{lemma}
\begin{proof} By Lemma \ref{Gperm}, we know that  either $\varphi$ induces automorphisms on $\calE_1$ and $\calE_2$, or $\varphi$ induces an isomorphism $\calE_1 \to \calE_2$. To see which occurs, we study the action on a general fibre of $\calE_1$.

So let $E_i$ denote the fibre of $\calE_i$ over a general point $p \in V$ and let $E_i'$ denote the fibre of $\calE_i$ over $\varphi(p)$ (for $i \in \{1,2\}$). By Proposition \ref{ellsurfprop}, we see that the $j$-invariants of $\{E_1,E_2\}$ are equal to those of $\{E_1',E_2'\}$. Thus, either
\begin{itemize}
\item $j(E_1) = j(E_1')$ then $j(E_2) = j(E_2')$, so $E_1 \cong E_1'$ and $E_2\cong E_2'$ and $\varphi$ induces automorphisms on $\calE_1$ and $\calE_2$, or
\item $j(E_1) = j(E_2')$, so $E_1 \cong E_2'$ and $\varphi$ induces an isomorphism $\calE_1 \to \calE_2$.
\end{itemize}
But the $j$-invariants of $E_i$ and $E_i'$ are given by the roots of Equation \eqref{jequation}, so the first (resp. second) case occurs if and only if the action of $\varphi$ preserves (resp. exchanges) these roots. \end{proof}

Let $H \subset G$ denote the subgroup of $G$ that preserves the decomposition of $\{F_3,\ldots,F_8\}$ into subsets $\{F_3,F_4,F_5\}$ and $\{F_6,F_7,F_8\}$. Then we can say more about the action of the subgroup $H$ on $\calE_1$ and $\calE_2$.

\begin{proposition} \label{Eiperm} \textup{(See \cite[Lemmas 7.6 and 7.7]{14thcase})} Let $\tau$ be any permutation in $H \subset S_6$ and let $\varphi$ denote the corresponding map on $C$.  Then
\begin{itemize}
\item If $\tau$ is an odd permutation, then $\varphi$ induces an isomorphism $\calE_1 \to \calE_2$.
\item If $\tau$ is an even permutation, then $\varphi$ induces automorphisms of $\calE_1$ and $\calE_2$
\end{itemize}\end{proposition}
\begin{proof} Suppose first that $\tau$ is an odd permutation.  Let $\gamma$ denote a path in $V$ that connects a point $p \in V$ to $\varphi(p)$. We will show that as we move along $\gamma$, the $j$-invariants of $E_1$ and $E_2$ are switched. To do this we use $f$ to push everything down to $B$. The image $f(\gamma)$ is a loop in $U$ starting and ending at $f(p)$. By Lemma \ref{lemma:jtracking}, we therefore need to show that monodromy around $f(\gamma)$ switches the roots of Equation \eqref{jequation}.

Monodromy around $f(\gamma)$ acts on the set of divisors $\{F_3,\ldots,F_8\}$ in the fibre $Y_{f(p)}$ of $\calY$ over $f(p)$ (see Section \ref{sect:Mpol}) as the permutation $\tau$. Furthermore, as $H$ is the subgroup of $G$ that preserves the sets $\{F_3,F_4,F_5\}$ and $\{F_6,F_7,F_8\}$, monodromy around $f(\gamma)$ must also preserve these sets.

As these divisors are permuted if and only if the roots of the cubic polynomials $(P(x) \pm 1)$ (see Equation \eqref{Pequation}) are permuted and as $\tau$ is an odd permutation, we see that the product of the discriminants of these cubics must vanish to odd order inside $\gamma$. This product is given by
\[\Delta := a^6 - 2a^3b^2 + b^4 - 2a^3 - 2b^2 +1,\]
where $a$ and $b$ are the $(a,b)$-parameters associated to the $M$-polarized fibre $X_{f(p)}$ of $\calX$ over $f(p)$.

Now, monodromy around $f(\gamma)$ switches the roots of Equation \eqref{jequation} if and only if its discriminant $(\sigma^2 - 4\pi)$ vanishes to odd order inside $\gamma$. However, by Remark \ref{distinctrem}, we find that this discriminant is given exactly by $\Delta$. So monodromy around $f(\gamma)$ switches the roots of Equation \eqref{jequation} and thus, by Lemma \ref{lemma:jtracking}, it induces an isomorphism $\calE_1 \to \calE_2$.

This completes the proof in the case when $\tau$ is an odd permutation. The proof when $\tau$ is an even permutation is similar.\end{proof}

To construct a model for $\calY_U$, our starting data consists of the cover $f\colon C \to B$ and the two elliptic surfaces $\calE_{1,2} \to C$. This data must satisfy the condition that the deck transformation group $G$ of the cover $f$ should act as automorphisms on or isomorphisms between the elliptic surfaces $\calE_1$ and $\calE_2$, in a way that is compatible with its action on $C$.

We begin by constructing a model for $\calY'_V$ by performing the Kummer construction fibrewise on $\calE_1 \times_C \calE_2$ to obtain a new threefold $\calY'$, which is isomorphic to $\calY'_V$ over $V$. Then, to obtain a model for $\calY_U$, we perform a quotient of $\calY'$ by $G$. However, the action of $G$ is not the obvious one induced by the action of $G$ on $\calE_1$ and $\calE_2$ (if it were, we would be able to undo the Kummer construction on $\calY_U$, which is not possible in general). Instead, we compose this action with the fibrewise automorphism induced by the action of $G$ (as a subset of $S_6$) on the set of curves $\{F_3,\ldots,F_8\}$.

\begin{remark} \label{dkrem} As the fibrewise Kummer construction defines a natural double Kummer pencil on smooth fibres of $\calF$, we can use the results of Kuwata and Shioda \cite[Section 5.2]{epdeefks} to define the elliptic fibration $\Psi$ on the smooth fibres of $\calF$. The curves $\{F_3,\ldots,F_8\}$ are then the components of the $I_2$ fibres that do not meet a chosen section.

Once the curves $\{F_3,\ldots,F_8\}$ are known, the automorphisms permuting them may be computed explicitly as compositions of the relevant symplectic automorphisms from \cite[Section 4]{agksaptec}.\end{remark}

Quotienting by this $G$-action, we obtain a new threefold $\calY \to B$. By construction, $\calY$ is isomorphic to $\calY_U$ over $U$, as required. We have a diagram:
\[\xymatrixcolsep{3pc}\xymatrix{ \calE_1 \times_C \calE_2 \ar@{-->}[r]^-{\mathrm{Kummer}} & \calY' \ar[r] & \calY & \calX \ar@{-->}[l]_{\mathrm{Nikulin}} \\
 \calA_V \ar[r] \ar@{^{(}->}[u] & \calY'_V \ar[r]\ar@{^{(}->}[u] & \calY_U \ar@{^{(}->}[u] & \calX_U  \ar@{-->}[l] \ar@{^{(}->}[u] 
}\]
This construction will be illustrated in the next section.

\section{Some Calabi-Yau threefolds fibred by Kummer surfaces}\label{CYsect}

In the remainder of this paper, we will illustrate how these methods can be used to construct explicit examples of Calabi-Yau threefolds. We note that, in this paper, a \emph{Calabi-Yau threefold} will always be a smooth projective threefold $\calX$ with $\omega_{\calX} \cong \calO_{\calX}$ and $H^1(\calX,\calO_{\calX})=0$. We further note that the cohomological condition in this definition implies that any fibration of a Calabi-Yau threefold by K3 surfaces must have base curve $\Proj^1$, so from this point we restrict our attention to the case where $B \cong \Proj^1$.

As our starting point, we will take the K3-fibred threefolds $\calX_g \to \Proj^1$ constructed in \cite{cytfmqk3s}. By construction, the N\'{e}ron-Severi group of a general fibre in these threefolds is isometric to $M_2 = H \oplus E_8 \oplus E_8 \oplus \langle -4 \rangle$, which admits an obvious embedding $M \hookrightarrow M_2$, and the restriction $\calX_U \to U$ of $\calX_g$ to the subset $U \subset \Proj^1$ over which the fibres are smooth is an $M_2$-polarized family of K3 surfaces. Thus these threefolds satisfy all of the conditions of Section \ref{backgroundsect}.

\subsection{A special family} \label{specialsect}

In \cite{cytfmqk3s}, the threefolds $\calX_g$ are constructed as resolved pull-backs of a special family $\calX_2 \to \calM_{M_2}$ over the (compact) $1$-dimensional moduli space $\calM_{M_2}$ of $M_2$-polarized K3 surfaces, by a map $g \colon \Proj^1 \to \calM_{M_2}$. To study the threefolds related to the $\calX_g$ by the construction detailed above, we will begin by studying $\calX_2$.

The family $\calX_2 \to \calM_{M_2}$ is described in \cite[Section 5.4.1]{flpk3sm}. It is given as the minimal resolution of the family of hypersurfaces in $\Proj^3$ obtained by varying $\lambda$ in the following expression
\begin{equation} \label{x2eq} \lambda w^4 + xyz(x+y+z-w) = 0.\end{equation}
This family has also been studied extensively by Narumiya and Shiga \cite{mmfk3sis3drp}, we will use some of their results in the sequel (note, however, that our $\lambda$ is not the same as the $\lambda$ used in \cite{mmfk3sis3drp}, instead, our $\lambda$ is equal to $\mu^4$ or $\frac{u}{256}$ from \cite{mmfk3sis3drp}).

Dolgachev \cite[Theorem 7.1]{mslpk3s} proved that $\calM_{M_2}$ is isomorphic to the compactification of the modular curve $\Gamma_0(2)^+ \setminus \mathbb{H}$. In \cite[Section 5.4.1]{flpk3sm} it is shown that the orbifold points of orders $(2,4,\infty)$ on this curve occur at $\lambda = (\frac{1}{256},\infty,0)$ respectively, and that the K3 fibres of $\calX_2$ are smooth away from these three points. Let $U_{M_2}$ denote the open set obtained from $\calM_{M_2}$ by removing these three points. Then the restriction $\calX_{2,U}$ of $\calX_2$ to $U_{M_2}$ is an $M_2$-polarized family of K3 surfaces, and $\calX_2$ satisfies all of the conditions of Section \ref{backgroundsect}.

We next compute the $(a,b,d)$-parameters of the fibres of $\calX_{2,U}$ (considered as $M$-polarized K3 surfaces, see Section \ref{sect:Mpol}). To do this, we use the fact that the standard and alternate fibrations on the K3 fibres are torically induced, so their $g_2$ and $g_3$ invariants may be computed (in terms of $\lambda$) using the toric geometry functionality of the computer software \emph{Sage}. These expressions can then be compared to the corresponding expressions computed for an $M$-polarized K3 surface in normal form (see, for instance, \cite[Theorem 3.1]{nfk3smmp}). We thus obtain
\begin{gather*}
a = \lambda + \frac{1}{2^43^2}
,\qquad
b = \frac{3}{8} \lambda - \frac{1}{2^63^3}
,\qquad
d = \lambda^3,
\end{gather*}
The $\sigma$ and $\pi$ invariants for the family $\calX_{2,U}$ are then given in terms of $\lambda$ by
\begin{align*}
\sigma &:= 2 - \frac{23}{2^63\lambda} + \frac{1}{2^63^3\lambda^2},\\
\pi &:= \left( 1 + \frac{1}{2^43^2\lambda}\right)^3.
\end{align*}

\subsection{Undoing the Kummer construction} \label{section:undokummer}

Let $\calY_{2,U} \to U_{M_2}$ denote the family of Kummer surfaces obtained from $\calX_2$ by quotienting by the Nikulin involution and resolving any resulting singularities. From Remark \ref{I2rem} and the $(a,b,d)$ parameters for $\calX_2$ computed above, we see that Assumption \ref{I2ass} is satisfied by $\calY_{2,U}$. We will explicitly show how to construct a model for $\calY_{2,U}$ from elliptic surfaces, as described in Section \ref{forwardconst}.

Our first step is to undo the Kummer construction for $\calY_{2,U} \to U_{M_2}$. To do this, we proceed to a cover $f\colon C_{M_2} \to \calM_{M_2}$, as computed in \cite[Section 5.3.2]{flpk3sm}. This cover is calculated in three steps. The first step is to take the cover $f_1\colon \Gamma_0(2) \setminus \mathbb{H} \to  \calM_{M_2}$. This is a double cover ramified over $\lambda \in \{\frac{1}{256},\infty\}$, which is given in coordinates by
\[ \lambda = -\mu^2 + \frac{1}{256},\]
where $\mu$ is a coordinate on $\Gamma_0(2) \setminus \mathbb{H}$ and the orbifold points of orders $(2,\infty,\infty)$ occur at $\mu = (\infty, \frac{1}{16},-\frac{1}{16})$ respectively.

The second step is to take the cover $f_2 \colon \Gamma_0(4) \setminus \mathbb{H} \to \Gamma_0(2) \setminus \mathbb{H}$. This is a double cover ramified over $\mu \in \{\frac{1}{16},\infty\}$, which is given in coordinates by
\[ \mu = -(\mu')^2 + \frac{1}{16},\]
where $\mu'$ is a coordinate on $\Gamma_0(4) \setminus \mathbb{H}$ and the three cusps occur at $\mu' = (0,\frac{1}{\sqrt{8}},-\frac{1}{\sqrt{8}})$.

Finally, the third step is to take the cover $f_3 \colon C_{M_2} \to \Gamma_0(4) \setminus \mathbb{H}$. This is a double cover ramified over $\mu' \in \{\frac{1}{\sqrt{8}},-\frac{1}{\sqrt{8}}\}$, which is given in coordinates by
\[ \mu' = \frac{1}{\sqrt{8}} \frac{(1-\nu^2)}{(1 + \nu^2)},\]
where $\nu$ is a coordinate on the rational curve $C_{M_2}$, which has four cusps occurring at $\nu \in \{0,1,-1,\infty\}$.

The composition of these three maps is the map $f\colon C_{M_2} \to \calM_{M_2}$ given in coordinates by
\[\lambda = \frac{1}{16} \frac{\nu^2(1-\nu^2)^2}{(1 + \nu^2)^4}.\]
This is an $8$-fold cover with deck transformation group $D_8$. It coincides precisely with the change of coordinates computed in \cite[Equation (7.1)]{mmfk3sis3drp}. Given this, the pulled-back family of Kummer surfaces over $C_{M_2}$ is given in affine coordinates by the expression $K(\nu)$ in \cite[Section 7]{mmfk3sis3drp}:
\begin{equation}\label{kummereq} u^2 = s(s-1)(s-\left(\frac{\nu+1}{\nu-1}\right)^2)t(t-1)(t-\nu^2)\end{equation}

We pause here to set up a little more notation. Let $V = f^{-1}(U_{M_2})$ and let $\calY'_{2,V}$ denote the pull-back of $\calY_{2,U}$ to $V$. By \cite[Propositions 4.1 and 5.6]{flpk3sm}, we may undo the Kummer construction for $\calY'_{2,V}$, so there is a family of Abelian surfaces $\calA_{2,V}$ that gives rise to $\calY'_{2,V}$ under fibrewise application of the Kummer construction. 

By Proposition \ref{ellsurfprop}, $\calA_{2,V}$ is isomorphic over $V$ to a fibre product $\calE_1 \times_{C_{M_2}} \calE_2$ of minimal elliptic surfaces $\calE_i \to C_{M_2}$ with section. These elliptic surfaces $\calE_i \to C_{M_2}$ are given in affine coordinates by the expressions $E_i(\nu)$ in \cite[Section 7]{mmfk3sis3drp}. We find
\begin{align*}
\calE_1\colon \quad z^2 &= t(t-1)(t-\nu^2), \\
\calE_2 \colon \quad w^2 &= s(s-1)(s-\left( \frac{\nu+1}{\nu-1}\right)^2).
\end{align*}
From this, we see that $\calE_2$ can be obtained from $\calE_1$ by applying the involution $\nu \mapsto \frac{\nu+1}{\nu-1}$. Thus, it suffices to study the elliptic surface $\calE_1$.

The $j$-invariant for $\calE_1$ is given by the expression (which can be computed directly, or by using Proposition \ref{ellsurfprop} and the fibrewise $\sigma$ and $\pi$ invariants for $\calX_{2,U}$ calculated in Section \ref{specialsect})
\[ j = \frac{4}{27}\frac{(\nu^4-\nu^2+1)^3}{\nu^4(\nu-1)^2(\nu+1)^2}.\]
The fibres of $\calE_1$ with $j$-invariants in the set $\{0,1,\infty\}$ are given in Table \ref{E1table}, where $\nu$ gives the location of the fibre in terms of the parameter $\nu$ on $C_{M_2}$, $j$ gives the corresponding value of the $j$-invariant, Multiplicity gives the order of vanishing of $j$ and Type gives the type of singular fibre. Finally, in the first column, $\omega$ denotes a primitive twelfth root of unity.

\begin{table}
\begin{center}
\begin{tabular}{|c|c|c|c|}
\hline
$\nu \in$ & $j$ & Multiplicity & Type\\ 
\hline
$\{1,-1\}$ & $\infty$ & $2$ & $I_2$ \\
$\{0,\infty\}$ & $\infty$ & $4$ & $I_4$ \\
$\{i,-i, \sqrt{2}, -\sqrt{2}, \frac{1}{\sqrt{2}}, -\frac{1}{\sqrt{2}}\}$ & $1$ & $2$ & $I_0$ \\
$\{\omega, -\omega, \frac{1}{\omega}, -\frac{1}{\omega}\}$ & $0$ & $3$ & $I_0$\\
\hline
\end{tabular}
\end{center}
\caption{Fibres of $\calE_1$ for $j \in \{0,1,\infty\}$.}
\label{E1table}
\end{table}

From this, we see that $\calE_1$ (and thus also $\calE_2$) is a rational elliptic surface with all fibres of type $I_n$ (in fact, it is even an \emph{extremal} rational elliptic surface, see \cite[Section VIII.1]{btes}). The fibre product $\calE_1 \times_{C_{M_2}} \calE_2$ is a singular threefold with isolated singularities in the fibres over $\nu \in \{0,1,-1,\infty\}$. Such threefolds have been studied by Schoen \cite[Lemma 3.1]{ofpress}, who showed that $\calE_1 \times_{C_{M_2}} \calE_2$ admits a small projective resolution. Denote such a resolution by $\calA_2$.

Our next task is to construct a model for $\calY_{2,V}'$. It follows from \cite[Lemma 6.1]{ofpress} that the involution defining the fibrewise Kummer construction on $\calE_1 \times_{C_{M_2}} \calE_2$ lifts to the resolution $\calA_2$. We can thus construct a model $\calY_2' \to C_{M_2}$ for $\calY_{2,V}'$ by performing the fibrewise Kummer construction on $\calA_2$. 

\subsection{Two special Calabi-Yau threefolds}\label{A2sect}

We digress briefly to study the properties of the threefolds $\calA_2$ and $\calY_2'$. The threefold $\calA_2$ is smooth and projective by construction and is isomorphic over $V$ to $\calA_{2,V}$. An easy application of adjunction shows that $\calA_2$ has trivial canonical bundle $\omega_{\calA_{2}} \cong \calO_{\calA_{2}}$ and, by \cite[Remark 2.1]{ofpress}, we see that it is also simply connected. $\calA_2$ is thus a Calabi-Yau threefold.

The invariants of $\calA_2$ can be computed using the methods of Schoen \cite{ofpress}. By the discussion in \cite[Section 2]{ofpress}, we see that the Euler characteristic $e(\calA_2)$ is equal to $64$. Furthermore, \cite[Proposition 7.1]{ofpress} implies that $\calA_2$ is a rigid threefold, so $h^{2,1}(\calA_2) = 0$. Thus we find that the remaining Hodge number $h^{1,1}(\calA_2) = 32$.

The threefold $\calY_2'$ is also smooth and projective, has trivial canonical bundle $\omega_{\calY'_2} \cong \calO_{\calY_2'}$ and, by \cite[Remark 6.5]{ofpress}, is simply connected. $\calY_2'$ is therefore also a Calabi-Yau threefold. Note that $\calY_2'$ is a resolution of the threefold given in affine coordinates by Equation \eqref{kummereq}.

The Euler characteristic of $\calY_2'$ can be computed using \cite[Lemma 6.3]{ofpress}, to obtain $e(\calY_2') = 80$. However, the Hodge numbers of $\calY'_2$ are not quite so simple to compute.

\begin{proposition} \label{Y2'hodge} $\calY_2'$ has Hodge numbers $h^{2,1}(\calY_2') = 0$ and $h^{1,1}(\calY'_2) = 40$.
\end{proposition}
\begin{proof} We compute $h^{2,1}(\calY_2')$ using the method of \cite[Section 2.1]{fpessakf}. From the affine form \eqref{kummereq}, we see that $\calY_2'$ is birational to a double cover of the weighted projective space $\WP(1,1,2,2)$ ramified over the pair of weighted cones given by
\begin{equation} \label{brancheq} s(s-(\mu+\nu)^2)(s-(\mu-\nu)^2)t(t-\mu^2)(t-\nu^2)=0,\end{equation}
where $(\mu,\nu,s,t)$ are homogeneous coordinates on $\WP(1,1,2,2)$ of weights $(1,1,2,2)$ respectively. We may therefore compute the space of deformations of $\calY_2'$ using the methods of \cite{iddcsav}.

However, in order to apply these methods we first need to ensure that our base space is smooth. Let $Z$ denote the blow-up of $\WP(1,1,2,2)$ along $\mu = \nu = 0$. Then $Z$ is a smooth variety. Let $B \subset Z$ denote the pull-back of the branch locus \eqref{brancheq} and let $(\hat{Z},\hat{B}) \to (Z,B)$ be a resolution of the singularities of $B$.  Then $\calY_2'$ is birational to a double cover of $\hat{Z}$ ramified over $\hat{B}$.

By \cite[Proposition 2.1]{iddcsav}, the space of deformations of $\calY_2'$ is isomorphic to
\[H^1(\hat{Z},\Theta_{\hat{Z}}(\log \hat{B})) \oplus H^1(\hat{Z},\Theta_{\hat{Z}} \otimes \calO_{\hat{Z}}(-6)),\] 
where $\Theta_{\hat{Z}}$ denotes the tangent bundle of $\hat{Z}$ and $\Theta_{\hat{Z}}(\log \hat{B})$ is the kernel of the natural restriction map $\Theta_{\hat{Z}} \to \calN_{\hat{B}|\hat{Z}}$. We will show that each of these direct summands vanishes, so that the space of deformations of $\calY_2'$ is trivial.

Since $Z$ is rigid (see \cite[Remark 2.21]{fpessakf}), applying \cite[Corollary 2]{iddcsav} we find that $H^1(\hat{Z},\Theta_{\hat{Z}}(\log \hat{B}))$ is isomorphic to the space of equisingular deformations of $B$ in $Z$. Furthermore, by \cite[Proposition 2.19]{fpessakf}, any equisingular deformation of $B$ in $Z$ induces a deformation of $\calA_2$. But we saw above that $\calA_2$ is rigid, so we must have $H^1(\hat{Z},\Theta_{\hat{Z}}(\log \hat{B})) = 0$.

To compute $H^1(\hat{Z},\Theta_{\hat{Z}} \otimes \calO_{\hat{Z}}(-6))$ we use \cite[Proposition 5.1]{iddcsav}. Suppose that the resolution $\hat{Z} \to Z$ is a composition of blow-ups along subvarieties $C_i$. Then \cite[Proposition 5.1]{iddcsav} gives an equation
\[h^1(\hat{Z},\Theta_{\hat{Z}} \otimes \calO_{\hat{Z}}(-6)) = h^1(Z,\Theta_{{Z}} \otimes \calO_{{Z}}(-6)) + \sum_{\mathrm{codim} C_i = 2} h^0(C_i,K_{C_i}). \]
Now, as $K_Z = \calO_Z(-6)$, we have
\[H^1(Z,\Theta_{{Z}} \otimes \calO_{{Z}}(-6)) \cong H^1(Z, \Omega_Z^{2}\otimes K_Z^{\vee} \otimes  \calO_Z(-6)) = H^1(Z,\Omega_Z^2),\]
which vanishes by \cite[(2.3)]{phnhvo}.

Thus to show that $H^1(\hat{Z},\Theta_{\hat{Z}} \otimes \calO_{\hat{Z}}(-6)) = 0$ it suffices to show that the curves $C_i$ are all rational. These curves arise from the multiple curves in the divisor $B$, which can be divided into three classes:
\begin{enumerate}
\item multiple curves arising from the preimages of the two points $(0,0,1,0)$ and $(0,0,0,1)$ under the blow-up $Z \to \WP(1,1,2,2)$,
\item the strict transforms in $Z$ of the double curves lying in the weighted cones $ s(s-(\mu+\nu)^2)(s-(\mu-\nu)^2) = 0$ and $t(t-\mu^2)(t-\nu^2) = 0$, and
\item the strict transforms in $Z$ of the double curves arising from the intersection between these two weighted cones.
\end{enumerate} 
The preimages in $Z$ of the two points $(0,0,1,0)$ and $(0,0,0,1)$ are copies of $\Proj^1$ and appear with multiplicity $3$ in $B$. To resolve the singularities of $B$, we blow up once along each copy of $\Proj^1$, then again along the six further $\Proj^1$'s arising as the intersection between the exceptional loci and the strict transform of $B$. As all curves blown up by this procedure are $\Proj^1$'s, they do not contribute to $h^1(\hat{Z},\Theta_{\hat{Z}} \otimes \calO_{\hat{Z}}(-6))$.

The curves of the second class are also all isomorphic to $\Proj^1$ and each gets blown up once to resolve $B$. Therefore they also do not contribute to $h^1(\hat{Z},\Theta_{\hat{Z}} \otimes \calO_{\hat{Z}}(-6))$.

The curves of the third class are the most difficult. There are nine of them and each gets blown up once to resolve $B$. To see that they are all rational, we note that each is a section of the fibration $Z \to \Proj^1$ given by projecting onto $(\mu,\nu)$. So they do not contribute to $h^1(\hat{Z},\Theta_{\hat{Z}} \otimes \calO_{\hat{Z}}(-6))$, which is therefore trivial.

Thus we find that the space of deformations of $\calY_2'$ is trivial. But this implies that $h^{2,1}(\calY_2') = 0$. Finally, the statement that $h^{1,1}(\calY_2') = 40$ follows immediately from the fact that $e(\calY_2') = 80$.
\end{proof} 

\begin{remark} We note that we can identify a lot of data about the Kummer fibration on $\calY_2'$ explicitly from $\hat{Z}$. This fibration is induced by the fibration $Z \to \Proj^1$ given by projecting onto $(\mu,\nu)$. The divisors $G_i$ in the double Kummer pencil on $\calY_2'$ are the lifts to $\calY_2'$ of the strict transforms of the three components of the cone $t(t-\mu^2)(t-\nu^2) = 0$ and the first exceptional divisor arising from the blow up of the curve in $Z$ over $(0,0,1,0)$, and the divisors $H_j$ are given by the same procedure applied to the cone $ s(s-(\mu+\nu)^2)(s-(\mu-\nu)^2) = 0$. Finally, the sixteen divisors $E_{ij}$ are given by the lifts to $\calY_2'$ of the nine exceptional divisors arising from blow-ups of curves of class (3), the six exceptional divisors arising from the second blow-up of curves of class (1), and the divisor over $\mu = \nu = 0$.
\end{remark}

\subsection{The action of {$D_8$}} \label{sect:D8action}

To get from $\calY'_2$ to a model for $\calY_{2,U}$ we need to understand the action of the group $D_8$ on $\calY_2'$. The action of $D_8$ on the base curve $C_{M_2}$ is generated by the automorphisms
\begin{align*}
\alpha\colon \nu & \longmapsto \frac{\nu - 1}{\nu + 1}, \\
\beta \colon \nu & \longmapsto -\nu,
\end{align*}
satisfying $\alpha^4 = \beta^2 = \mathrm{Id}$ and $\beta\circ\alpha\circ\beta = \alpha^{-1}$. Note that the automorphism $\alpha$ interchanges $\calE_1$ and $\calE_2$, whilst $\beta$ preserves them.

We next compute the action of $\alpha$ and $\beta$ on the threefold $\calY_2'$. To do this, we need to understand how these automorphisms act on the fibres of $\calY_2'$. But this action is identical to that induced on the fibres of $\calY_{2,U}$ by monodromy around appropriately chosen loops in $U_{M_2}$.

By the results of \cite[Section 4.3]{flpk3sm}, the action of monodromy around loops in $U_{M_2}$ on the fibres of $\calY_{2,U}$ is completely encoded by its action on the divisors $\{F_3,\ldots,F_8\}$ . This action may be calculated using the fact (see Section \ref{sect:Mpol}) that these divisors lie in the six $I_2$ fibres of the elliptic fibration $\Psi$ on the K3 fibres of $\calY_{2,U}$, the locations of which are given by the roots of the polynomials $(P(x) \pm 1)$, where $P(x)$ is as defined in Equation \eqref{Pequation} (and $(a,b)$ in this equation are the $(a,b)$ values associated to the $M$-polarized fibres of $\calX_{2,U}$, calculated in Section  \ref{specialsect}, normalized so that $d=1$).

Thus, to compute the action of monodromy on the divisors $\{F_3,\ldots,F_8\}$, it suffices to compute its action on the roots of $(P(x) \pm 1)$. This action may be calculated explicitly using the \texttt{monodromy} command in \emph{Maple}'s \texttt{algcurves} package, with the base point $\lambda = -\frac{257}{256}$ (which was chosen for ease of lifting to the covers of $U_{M_2}$). We obtain Table \ref{monodromy}, where $\lambda$ gives the $\lambda$-value at the puncture in $U_{M_2}$ around which monodromy occurs and Monodromy gives the action of anticlockwise monodromy around that point on the divisors  $\{F_3,\ldots,F_8\}$ in the fibre over $\lambda = -\frac{257}{256}$, expressed as a permutation in $S_6$, where we assign the labels $1,\ldots,6$ to $F_3,\ldots,F_8$ respectively.

\begin{table}
\begin{center}
\begin{tabular}{|c|c|}
\hline
$\lambda$ & Monodromy \\ 
\hline
$0$ & $(14)(25)(36)$ \\
$\frac{1}{256}$ & $(12)$ \\
$\infty$ & $(1524)(36)$ \\
\hline
\end{tabular}
\end{center}
\caption{Action of monodromy around punctures in $U_{M_2}$ on $\{F_3,\ldots,F_8\}$.}
\label{monodromy}
\end{table}

Now we lift this information to $\calY'_2$. There is a natural double Kummer pencil on the fibres of $\calY'_{2,V}$, defined by the fibrewise Kummer construction, so we may use the method of Remark \ref{dkrem} to define the elliptic fibration $\Psi$ and the divisors $\{F_3,\ldots,F_8\}$ on the K3 fibres of $\calY'_{2,V}$. We can then compute the action of the automorphisms $\alpha$ and $\beta$ on the divisors $\{F_3,\ldots,F_8\}$ (calculated using the base point $\nu = -\frac{\sqrt{17}}{4}$). We find that $\alpha$ acts on these divisors as the permutation $(1524)(36)$ and $\beta$ acts as the permutation $(14)(25)(36)$.

With this in place, we are now ready to compute the action of $\alpha$ and $\beta$ on the threefold $\calY_2'$. The expressions we compute will use the coordinates $(\nu,s,t,u)$ from the affine description given in Equation \eqref{kummereq}.

We begin with the involution $\beta$. From the calculations above, $\beta$ should be induced by $\nu \mapsto -\nu$ on the base $C_{M_2}$, so fixes the points $\nu \in \{0,\infty\}$. The monodromy calculations above show that $\beta$ should exchange the divisors $F_i \leftrightarrow F_{i+3}$ (for $i \in \{3,4,5\}$) in the fibres of $\calY'_2$. It is easy to show that this action is realised by the involution $\beta$ given by
\[\beta\colon (\nu,s,t,u) \longmapsto \left(-\nu,\left( \frac{\nu-1}{\nu+1} \right)^2s,t,\left( \frac{\nu-1}{\nu+1} \right)^3u\right).\]

Instead of computing the action of $\alpha$ directly, we will instead compute the action of the involution $\alpha\circ\beta$. This involution should be induced by $\nu \mapsto \frac{\nu+1}{\nu-1}$ on the base $C_{M_2}$, so fixes the points $\nu = 1 \pm \sqrt{2}$. The monodromy calculations above show that $\alpha\circ\beta$ should fix the divisors $\{F_3,F_4,F_5,F_8\}$ and exchange $F_6 \leftrightarrow F_{7}$ in the fibres of $\calY'_2$. 

If we compute the involution induced on $\calY'_2$ by $\nu \mapsto \frac{\nu+1}{\nu-1}$, we obtain
\[\iota\colon (\nu,s,t,u) \longmapsto \left( \frac{\nu+1}{\nu-1},t,s,u\right).\]
However, $\iota$ cannot be equal to $\alpha\circ\beta$: it does not preserve the fibration $\Psi$ on the fibres of $\calY'_2$ (as it switches the divisors $G_i$ and $H_i$ from the double Kummer pencil, which has the effect of permuting the sections of $\Psi$) and it does not exchange the divisors $F_6 \leftrightarrow F_{7}$. Both of these problems can be rectified by composing with the unique fibrewise symplectic automorphism $\varphi$ from \cite[Section 4.1]{agksaptec} that exchanges $G_i \leftrightarrow H_i$ and $F_6 \leftrightarrow F_7$ (which may be computed by the method of \cite[Example 7.9]{14thcase}).

Thus, we find that $\alpha\circ\beta = \varphi \circ \iota$. Composing with $\beta$, we find that $\alpha$ is given as the composition $\varphi \circ \iota'$, where
\[\iota'\colon (\nu,s,t,u) \longmapsto \left(\frac{\nu -1}{\nu + 1},\frac{t}{\nu^2},s,\frac{u}{\nu^3}\right)\]
and $\varphi$ is as before.

Define $\calY_2$ to be the threefold obtained by quotienting $\calY_2'$ by the action of $D_8$ described above and resolving any resulting singularities. Then $\calY_2$ is isomorphic to $\calY_{2,U}$ over $U_{M_2}$ by construction, so gives the required model for $\calY_{2,U}$. These threefolds fit together in a diagram:
\[\xymatrixcolsep{2.5pc}\xymatrix{ \calE_1 \times_{C_{M_2}} \calE_2 \ar[d] & \calA_2 \ar[l]_-{\mathrm{resolve}} \ar@{-->}[r]^{\mathrm{Kummer}} \ar[d] & \calY_2' \ar@{-->}[r] \ar[d] & \calY_2 \ar[d] & \calX_2 \ar@{-->}[l]_{\mathrm{Nikulin}} \ar[d] \\
C_{M_2} \ar@{=}[r] & C_{M_2} \ar@{=}[r] & C_{M_2} \ar[r]^{f} & \calM_{M_2} \ar@{=}[r] & \calM_{M_2}
}\]

\subsection{Singular fibres}\label{singfibsect}

Our next task is to study the forms of the singular fibres in the threefolds from the above diagram. These singular fibres come in three flavours, lying over $\lambda \in \{0,\frac{1}{256},\infty\} \subset \calM_{M_2}$; we will study each in turn.

\subsubsection{Fibres over $\lambda = 0$}\label{0subsect}

The first singular fibres we will study lie over the point $\lambda = 0$, which is the cusp in the orbifold base curve. Let $\Delta_0$ denote a small disc around $\lambda = 0$ and let $\Delta'_0$ denote one of the connected components of the preimage of $\Delta$ under $f$. The map $\Delta'_0 \to \Delta_0$ is a double cover ramified over $\lambda = 0 \in \Delta_0$. The preimages of $\lambda = 0$ are $\nu \in \{0,1,-1,\infty\}$; without loss of generality we will assume that $\Delta'_0$ is the connected component containing $\nu = 0$, the other choices give identical results.

The singular fibre of $\calE_1  \times_{C_{M_2}} \calE_2$ over $\nu = 0$ is a product $I_2 \times I_4$ of singular elliptic curves and has $8$ components, each of which is isomorphic to $\Proj^1 \times \Proj^1$. The threefold $\calE_1  \times_{C_{M_2}} \calE_2$ has eight nodes occurring at the points where four components of the central fibre intersect; these may be resolved by a small projective resolution \cite[Lemma 3.1]{ofpress} to give the threefold $\calA_2$, which is smooth. The singular fibre of $\calA_2$ over $\nu=0$ again has $8$ components, which are rational surfaces.

The involution defining the Kummer construction fixes four of the components of the $\nu=0$ singular fibre of $\calA_2$ and acts to exchange the other four as two pairs. The resultant singular fibre of $\calY_2'$ has six components, each of which is a rational surface, arranged in a cube and the threefold $\calY_2'$ is again smooth.

Finally, the involution defining the quotient $\calY_2' \to \calY_2$ is given by $\beta$ from Section \ref{sect:D8action}. This involution acts trivially on two of the components of the $\nu = 0$ fibre in $\calY_2'$ (the ``top'' and ``bottom'' of the cube) and acts as an involution on each of the remaining four (the ``sides'' of the cube). Upon performing the quotient, these four components become exceptional and may be contracted, giving a singular fibre consisting of two rational components in the threefold $\calY_2$ (which is smooth over $\Delta_0$). These rational components meet along four rational curves, which form a cycle in each component.

\subsubsection{Fibres over $\lambda = \frac{1}{256}$}

The next type of singular fibres lie over the point $\lambda = \frac{1}{256}$, which is an orbifold point of order $2$ in the base curve. As before, let $\Delta_{\frac{1}{256}}$ denote a small disc around $\lambda = \frac{1}{256}$ and let $\Delta'_{\frac{1}{256}}$ denote one of the connected components of the preimage of $\Delta_{\frac{1}{256}}$ under $f$. The map $\Delta'_{\frac{1}{256}} \to \Delta_{\frac{1}{256}}$ is a double cover ramified over $\lambda = \frac{1}{256} \in \Delta$. The preimages of $\lambda = \frac{1}{256}$ are $\nu \in \{1 \pm \sqrt{2},-1\pm\sqrt{2}\}$; without loss of generality we will assume that $\Delta'_{\frac{1}{256}}$ is the connected component containing $\nu = 1 + \sqrt{2}$, the other choices give identical results.

The fibre of $\calE_1  \times_{C_{M_2}} \calE_2$ over $\nu = 1+\sqrt{2}$ is a product $E \times E$ of a smooth elliptic curve $E$ with $j(E) = \frac{125}{27}$ with itself, so the threefold $\calE_1  \times_{C_{M_2}} \calE_2$ is smooth and isomorphic to $\calA_2$ over $\Delta'_{\frac{1}{256}}$. The fibre of $\calY_2'$ over $\nu = 1+\sqrt{2}$ is thus just the Kummer surface associated to a product of an elliptic curve with itself and the threefold $\calY_2'$ is smooth.

The involution defining the quotient $\calY_2' \to \calY_2$ is given by $\beta\alpha  = \varphi \circ \iota$ from Section \ref{sect:D8action}. Its action is calculated in the same way as \cite[Example 7.9]{14thcase} and yields the same result: for the action of the fibrewise involution $\varphi$ to be well-defined we must first contract two curves lying in the fibre of $\calY_2'$ over $\nu=1 + \sqrt{2}$, giving two nodes in the threefold total space. After this contraction has been performed the involution $\beta\alpha$ is well-defined and acts trivially on the fibre over $\nu = 1 + \sqrt{2}$.

We find that the threefold $\calY_2$ is smooth over $\Delta_{\frac{1}{256}}$ and its fibre over $\lambda = \frac{1}{256}$ is a singular K3 surface containing two $A_1$ singularities. Under the double cover ramified over the fibre $\lambda = \frac{1}{256}$ these become two nodes in the threefold total space, which may be crepantly resolved to give the threefold $\calY_2'$.

\subsubsection{Fibres over $\lambda = \infty$}\label{infinitysubsect}

The final, and most difficult, type of singular fibres lie over the point $\lambda = \infty$, which is an orbifold point of order $4$ in the base curve. As before, let $\Delta_{\infty}$ denote a small disc around $\lambda = \infty$ and let $\Delta'_{\infty}$ denote one of the connected components of the preimage of $\Delta_{\infty}$ under $f$. The map $\Delta'_{\infty} \to \Delta_{\infty}$ is a cyclic four-fold cover ramified over $\lambda = \infty \in \Delta_{\infty}$. The preimages of $\lambda = \infty$ are $\nu = \pm i$; without loss of generality we will assume that $\Delta'_{\infty}$ is the connected component containing $\nu = i$, the other choice gives the same result.

The fibre of $\calE_1  \times_{C_{M_2}} \calE_2$ over $\nu = i$ is a product $E \times E$ of a smooth elliptic curve $E$ with $j(E) = 1$ with itself, so the threefold $\calE_1  \times_{C_{M_2}} \calE_2$ is smooth and isomorphic to $\calA_2$ over $\Delta'_{\infty}$. The fibre of $\calY_2'$ over $\nu =i$ is thus just the Kummer surface associated to a product of an elliptic curve with itself and the threefold $\calY_2'$ is smooth.

We compute the fibre of $\calY_2$ over $\lambda = \infty$ in two stages. First, we compute the quotient of $\calY_2'$ by the involution $\alpha^2$. This involution acts on the fibre over $\nu = i$ as $u \mapsto -u$, which fixes the four curves $G_i$ and $H_j$. The quotient by $\alpha^2$ therefore has a fibre of multiplicity $2$ and eight disjoint curves of $cA_1$ singularities, given by the images of the $G_i$ and $H_j$. These curves may be crepantly resolved to give $8$ exceptional divisors. The resulting threefold is smooth and has a singular fibre with $9$ components: a rational component of multiplicity $2$, isomorphic to $\Proj^1 \times \Proj^1$ blown up at sixteen points (the sixteen $(-1)$-curves are the images of the $E_{ij}$), meeting eight disjoint exceptional components of multiplicity $1$, each of which is isomorphic to $\mathbb{F}_2$. It is an example of a ``flowerpot degeneration'' \cite{bgd9} with flowers of type $4\alpha$ (see \cite[Table 3.3]{bgd9}).

Next, we compute the quotient of this threefold by $\alpha$, which now acts as an involution. This involution acts on the images of the $G_i$ and $H_j$ as follows: $\alpha$ fixes $H_0$ and $H_1$ pointwise, exchanges $H_2$ and $H_3$, and acts on each $G_i$ as an involution that fixes the intersections between $G_i$ and the images of $E_{i0}$ and $E_{i1}$.

Thus, on the degenerate fibre, $\alpha$ acts as an involution on each of the components over the curves $G_i$, $H_0$ and $H_1$, and exchanges the components over $H_2$ and $H_3$. This action has twelve disjoint fixed curves: the curves $H_0$ and $H_1$, the eight fibres in the $\mathbb{F}_2$-components over the $G_i$ that lie above the intersections $G_i \cap E_{i0}$ and $G_i \cap E_{i1}$ (two curves in each of four components), and the $(-2)$-sections in the $\mathbb{F}_2$-components over $H_0$ and $H_1$. 

The degenerate fibre in the quotient by $\alpha$ therefore has eight components: one of multiplicity $4$, coming from the quotient of the component of multiplicity $2$, and seven of multiplicity $2$, coming from the quotients of the eight components of multiplicity $1$ (recall that the components over $H_2$ and $H_3$ are identified). Furthermore, the $12$ fixed curves give rise to $12$ disjoint curves of $cA_1$ singularities in the threefold total space, which may be crepantly resolved to obtain a further $12$ exceptional components.

After performing this resolution, we find that the threefold $\calY_2$ is smooth over $\Delta_{\infty}$. Its fibre over $\lambda = \infty$ has $20$ components: one of multiplicity $4$, two of multiplicity $3$, seven of multiplicity $2$ and ten of multiplicity $1$.
\medskip

From this calculation and the adjunction formula for multiple covers, it is easy to see that:

\begin{theorem} $\calY_2 \to \calM_{M_2}$ is a smooth threefold fibred by Kummer surfaces with canonical bundle $\omega_{\calY_2} \cong \calO_{\calY_2}(-F)$, where $F$ is the class of a K3 surface fibre. Moreover, the restriction $\calY_{2,U}$ of $\calY_2$ to the open set $U_{M_2}$ is isomorphic to the resolved quotient of $\calX_{2,U}$ by the fibrewise Nikulin involution.
\end{theorem}

\subsection{Constructing Calabi-Yau threefolds}\label{cy3constsect}

Recall that, in \cite{cytfmqk3s}, the threefolds $\calX_g$ were constructed by pulling-back the family $\calX_2 \to \calM_{M_2}$ by a map $g\colon \Proj^1 \to \calM_{M_2}$ and resolving singularities. The aim of this section is to perform the same construction with the threefold $\calY_2$, then to see how the resulting threefolds are related to the $\calX_g$.

As in \cite{cytfmqk3s}, let $g\colon \Proj^1 \to \calM_{M_2}$ be an $n$-fold cover and let $[x_1,\ldots,x_k]$, $[y_1,\ldots,y_l]$ and $[z_1,\ldots,z_m]$ be partitions of $n$ encoding the ramification profiles of $g$ over $\lambda = 0$, $\lambda = \infty$ and $\lambda = \frac{1}{256}$ respectively. Let $r$ denote the degree of ramification of $g$ away from $\lambda \in \{0,\frac{1}{256},\infty\}$, defined to be
\[r := \sum_{\mathclap{\substack{p \in \Proj^1 \\ g(p) \notin  \{0,\frac{1}{256},\infty\}}}} (e_p -1),\]
where $e_p$ denotes the ramification index of $g$ at the point $p \in \Proj^1$.

Now let $\bar{\calY}_2$ denote the threefold obtained from $\calY_2$ by contracting all of the components in the fibre over $\lambda = \infty$ that have multiplicity less than $4$ (in a neighbourhood of $\lambda = \infty$, the threefold $\bar{\calY}_2$ is isomorphic to the quotient of $\calY_2'$ by the action of $D_8$). Let $\bar{\psi}_g\colon \bar{\calY}_g \to \Proj^1$ denote the normalization of the pull-back $g^*(\bar{\calY}_2)$. Then we have the following analogue of \cite[Proposition 2.3]{cytfmqk3s}.

\begin{proposition} \label{Ytrivialcanonicalprop} The threefold $\bar{\calY}_g$ has trivial canonical sheaf if and only if  $k+l+m-n-r=2$ and either $l = 2$ with $y_1,y_2 \in \{1,2,4\}$, or $l = 1$ with $y_1 = 8$.
\end{proposition}
\begin{proof} This is proved in exactly the same way as  \cite[Proposition 2.3]{cytfmqk3s}. \end{proof}

Next we prove an analogue of \cite[Proposition 2.4]{cytfmqk3s}.

\begin{proposition} \label{KYgtrivial} If Proposition \ref{Ytrivialcanonicalprop} holds, then there exists a projective birational morphism $\calY_g \to \bar{\calY}_g$, where ${\calY}_g$ is a normal threefold with trivial canonical sheaf and at worst $\Q$-factorial terminal singularities. Furthermore, any singularities of $\calY_g$ occur in its fibres over $g^{-1}(\frac{1}{256})$, and ${\calY}_g$ is smooth if $g$ is unramified over $\lambda = \frac{1}{256}$ \textup{(}which happens if and only if $m = n$\textup{)}. 
\end{proposition}

\begin{proof} We follow the same method that was used to prove \cite[Proposition 2.4]{cytfmqk3s} and show that the singularities of $\bar{\calY}_g$ may all be crepantly resolved, with the possible exception of some $\Q$-factorial terminal singularities lying in fibres over $g^{-1}(\frac{1}{256})$.

First note that the threefold $\bar{\calY}_g$ is smooth away from the fibres lying over $g^{-1}\{0,\frac{1}{256},\infty\}$, so it suffices to compute crepant resolutions in a neighbourhood of each of these fibres.

First let $\Delta_{\infty}$ denote a disc in $\calM_{M_2}$ around $\lambda = \infty$ and let $\Delta'_{\infty}$ denote one of its preimages under $g$. Then $g\colon \Delta'_{\infty} \to \Delta_{\infty}$ is a $y_i$-fold cover ramified totally over $\lambda = \infty$, for some $y_i \in \{1,2,4,8\}$.

However, in the three cases $y_i \in \{1,2,4\}$, over $\Delta_{\infty}$ we have that $\bar{\calY}_g$ is isomorphic to a quotient of $\calY_2'$.  Crepant resolutions of such quotients were computed in Subsection \ref{infinitysubsect}: these resolutions have $20$, $9$ and $1$ components in the cases $y_i = 1$, $2$ and $4$ respectively. The case $y_i = 8$ is a double cover of the case $y_i = 4$, in this case $\bar{\calY}_g$ is smooth with one component.

Next let $\Delta_0$ denote a disc in $\calM_{M_2}$ around $\lambda = 0$ and let $\Delta'_0$ denote one of its preimages under $g$. Then $g\colon \Delta'_0 \to \Delta_0$ is an $x_i$-fold cover ramified totally over $\lambda = 0$, for some $x_i$. The fibre of $\bar{\calY}_2$ over $\lambda = 0$ was computed in Subsection \ref{0subsect}: it consists of two rational surfaces meeting along four rational curves $D_1,\ldots,D_4$, which are arranged in a cycle in each component.

Upon proceeding to the $x_i$-fold cover, we find that the threefold $\bar{\calY}_g$ contains four curves of $cA_{x_i-1}$ singularities in its fibre over $g^{-1}(0)$, given by the pull-backs of the curves $D_i$. This has a crepant resolution which contains:
\begin{itemize}
\item $2$ components that are strict transforms of the original $2$,
\item $4(x_i-1)$ components arising from the blow-ups of the four curves of $cA_{x_i-1}$ singularities lying over the $D_i$, and
\item $(x_i - 2)^2$ (if $x_i$ is even) or $(x_i -2)^2 - 1$ (if $x_i$ is odd) components arising from the blow-ups of the intersections between these four curves.
\end{itemize}

To see this, assume first that $x_i$ is even. Then we may factorize the map  $g\colon \Delta'_0 \to \Delta_0$ into an $\frac{x_i}{2}$-fold cover followed by a double cover. We compute a crepant resolution of $\bar{\calY}_g$ over $ \Delta'_0$ as follows. First pull-back $\bar{\calY}_2$ to a double cover of $\Delta_0$ ramified over $\lambda = 0$. The resulting threefold has a crepant resolution, which is locally isomorphic to $\calY_2'$ (in a neighbourhood of $\nu = 0$). Then pull-back again to an $\frac{x_i}{2}$-fold cover ramified over the preimage of $\lambda = 0$. As the fibre of $\calY_2'$ over $\nu = 0$ is semistable, consisting of six rational components arranged in a cube (by Subsection \ref{0subsect}), we may compute a crepant resolution of this $\frac{x_i}{2}$-fold cover using results of Friedman \cite[Section 1]{bgd7}. This gives a crepant resolution of $\bar{\calY}_g$ with the required properties.

Next, assume that $x_i$ is odd. Consider the $2x_i$-fold cover $g'\colon \Delta''_0 \to \Delta_0$ that is ramified totally over $\lambda = 0$. Then a crepant resolution $\calY_{g'}$ of the pull-back of $\bar{\calY}_2$ by $g'$ can be computed as above. Furthermore, there is an involution on this resolution, the quotient by which gives a threefold birational to $\bar{\calY}_g$.

This involution preserves every component of the fibre of $\calY_{g'}$ over $(g')^{-1}(0)$. Its fixed locus consists of the strict transforms of the two components of $\bar{\calY}_2$, along with $(x_i - 1)$ of the exceptional components arising from the blow-up of each curve of $cA_{x_i-1}$ singularities and $(\frac{x_i}{2}-1)^2-\frac{1}{4}$ of the exceptional components arising from the blow-up of each of their intersections. The components appearing in this fixed locus are uniquely determined by the properties that no two of them meet in a double curve and that every non-fixed component meets precisely two fixed ones.

Under the quotient, the non-fixed components become exceptional and may be contracted, resulting in a smooth threefold that resolves $\bar{\calY}_g$. This resolution is crepant by the adjunction formula for double covers.

Finally, let $\Delta_{\frac{1}{256}}$ be a disc in $\calM_{M_2}$ around $\lambda = \frac{1}{256}$ and let $\Delta'_{\frac{1}{256}}$ be one of the connected components of its preimage under $g$. Then $g\colon \Delta'_{\frac{1}{256}} \to \Delta_{\frac{1}{256}}$ is a $z_i$-fold cover ramified totally over $\lambda = \frac{1}{256}$, for some $z_i$. 

The threefold $\bar{\calY}_2$ is smooth over $\Delta_{\frac{1}{256}}$, but its fibre over $\lambda = \frac{1}{256}$ has two isolated $A_1$ singularities. Upon proceeding to the $z_i$-fold cover $\Delta'_{\frac{1}{256}} \to \Delta_{\frac{1}{256}}$, these become a pair of isolated terminal singularities of type $cA_{z_i-1}$ in $\bar{\calY}_g$. 

Thus $\calY_g$ is smooth away from its fibres over $g^{-1}(\frac{1}{256})$, where it can have isolated terminal singularities. By \cite[Theorem 6.25]{bgav}, we may further assume that $\calY_g$ is $\Q$-factorial. To complete the proof, we note that if $g$ is a local isomorphism over $\Delta_{\frac{1}{256}}$, then $\calY_g$ is also smooth over $g^{-1}(\frac{1}{256})$ and thus smooth everywhere.
\end{proof}

Let $\psi_g\colon \calY_g \to \Proj^1$ denote the fibration induced on $\calY_g$ by the map $\bar{\psi}_g\colon \bar{\calY}_g \to \Proj^1$. Then $\calY_g$ is a threefold fibred by Kummer surfaces. It follows that:

\begin{proposition} \label{YgCY3}  Let $\calY_g$ be a threefold as in Proposition \ref{KYgtrivial}. If $\calY_g$ is smooth, then $\calY_g$ is a Calabi-Yau threefold.
\end{proposition}
\begin{proof} By Proposition \ref{KYgtrivial}, we see that $\calY_g$ has trivial canonical bundle. The condition on the vanishing of the first cohomology $H^1(\calY_g,\calO_{\calY_g})$ is proved in exactly the same way as \cite[Proposition 2.6]{cytfmqk3s}.
\end{proof}

This proposition enables the construction of many Kummer surface fibred Calabi-Yau threefolds $\psi_g\colon \calY_g \to \Proj^1$. These Calabi-Yau threefolds are related to the Shioda-Inose fibred threefolds $\pi_g\colon \calX_g \to \Proj^1$ constructed in \cite{cytfmqk3s} as follows. Let $U := g^{-1}(U_{M_2})$. Then, by construction, the restriction $\psi_g\colon \calY_g|_{U} \to U$ of $\calY_g$ to $U$ is isomorphic to the threefold obtained from the restriction $\pi_g\colon \calX_g|_{U} \to U$ by quotienting by the fibrewise Nikulin involution and resolving singularities (note that the maps $g\colon \Proj^1 \to \calM_{M_2}$ defining $\calX_g$ and $\calY_g$ here are the same). So we may think of the $\calY_g$ as arising from the $\calX_g$ through the process described in Section \ref{sect:threefoldkummer}.

\begin{remark} In fact, there is a kind of duality between the Shioda-Inose fibred threefolds $\pi_g\colon \calX_g \to \Proj^1$ and the Kummer fibred threefolds $\Psi_g\colon \calY_g \to \Proj^1$. As noted above, for $U := g^{-1}(U_{M_2})$, the restriction $\calY_g|_U$ is isomorphic to the resolved quotient of $\calX_g|_U$ by the fibrewise Nikulin involution. However, $\calX_g|_U$ is \emph{also} isomorphic to the resolved quotient of $\calY_g|_U$ by the fibrewise Nikulin involution given in Remark \ref{Nikulinrem2}. So we can move back and forth between $\calX_g|_U$ and $\calY_g|_U$ by quotienting by Nikulin involutions and resolving singularities.
\end{remark}

\subsection{Properties of the constructed threefolds}\label{propertiessect}

The properties of the threefolds $\calY_g$ are closely linked to those of the related threefolds $\calX_g$. Of particular importance to these calculations is the curve $C_g \subset \calX_g$, defined as the closure of the fixed locus of the fibrewise Nikulin involution on $\calX_g|_U$. As the Nikulin involution has $8$ fixed points in a general fibre of $\calX_g$, the curve $C_g$ is an $8$-fold cover of $\Proj^1$.

The curve $C_g$ is easily calculated as the pull-back of the curve $C_2 \subset \calX_2$ (defined in the same way as $C_g \subset \calX_g$) by the map $g$. The properties of $C_2$ are as follows:

\begin{lemma} \label{C2props} The curve $C_2 \subset \calX_2$ has three irreducible components, all of which have genus $0$. Two of these components are double covers of $\calM_{M_2}$ ramified over $\lambda \in \{0, \infty\}$. The third component is a $4$-fold cover of $\calM_{M_2}$ that has ramification profile $[2,1,1]$ over $\lambda = \frac{1}{256}$, ramification profile $[2,2]$ over $\lambda = 0$, and ramification profile $[4]$ over $\lambda = \infty$.
\end{lemma}
\begin{proof} Let $p \in U_{M_2}$ be a general point and let $X_p$ (resp. $Y_p$) be the fibre of $\calX_2$ (resp. $\calY_2$) over $p$. Define the divisors $\{F_1,\ldots,F_8\}$ in $Y_p$ as in Section \ref{sect:Mpol}. The $F_i$ arise as the exceptional curves in the resolution of $X_p/\beta$, where $\beta$ is the Nikulin involution on $X_p$. Thus, the action of monodromy in $\pi_1(U_{M_2},p)$ on the $8$ fixed points of $\beta$ in $X_p$, which determines the curve $C_2$, is the same as the action of monodromy on the divisors $F_i$ in $Y_p$.

The action of monodromy on the divisors $\{F_3,\ldots,F_8\}$ was computed explicitly in Table \ref{monodromy}. Using this, the action of monodromy on $\{F_1,F_2\}$ may be computed from \cite[Proposition 4.8]{flpk3sm}. From this it is easy to compute the description of the ramification profiles of the components of $C_2$, their genera may be computed by Hurwitz's theorem.
\end{proof}

It turns out that many of the properties of the Calabi-Yau threefolds $\calY_g$ can be calculated from knowledge of the curve $C_g$ and the ramification behaviour of the map $g$. At this point we restrict ourselves to the case $l = 2$, to avoid pathologies occurring when $l = 1$ (see \cite[Remark 3.1]{cytfmqk3s}).

\begin{proposition} \label{Ygh11} Let $\calY_g$ be a Calabi-Yau threefold as in Proposition \ref{YgCY3} and suppose that $g^{-1}(\infty)$ consists of two points  \textup{(}so that $l=2$\textup{)}. Then
\[h^{1,1}(\calY_g) = 12 + \sum_{x_i\ \mathrm{odd}} x_i^2 + \sum_{x_i\ \mathrm{even}} (x_i^2 + 1) + s + c_1 + c_2,\]
where $[x_1,\ldots,x_k]$ is the partition of $n$ encoding the ramification profile of $g$ over $\lambda = 0$, $s$ is the number of irreducible components of $C_g$, and $c_1$, $c_2$ are given in terms of the partition $[y_1,y_2]$ of $n$ encoding the ramification profile of $g$ over $\lambda = \infty$ by $c_j = 19$ \textup{(}resp. $8$, $0$\textup{)} if and only if $y_j = 1$ \textup{(}resp. $2$, $4$\textup{)}.
\end{proposition}

\begin{proof} We follow the same method used to prove \cite[Proposition 3.2]{cytfmqk3s}, by noting that $h^{1,1}(\calY_g)$ is equal to the sum of the ranks of the groups of horizontal divisors $\Pic^h(\calY_g)$ and vertical divisors $\Pic^v(\calY_g)$.

We begin with the subspace of horizontal divisors. As before, we have an embedding $\Pic^h(\calY_g) \hookrightarrow \Pic(Y)$, where $Y$ denotes a general fibre of $\calY_g$, given by restriction. Furthermore, by \cite[Corollary 3.2]{flpk3sm}, monodromy around singular fibres can only act non-trivially on the $8$-dimensional sublattice of $\Pic(Y)$ generated by the eight curves $\{F_1,\ldots,F_8\}$ (defined as in Section \ref{sect:Mpol}). Thus, every divisor in the $11$-dimensional orthogonal complement to this set is preserved under monodromy, so sweeps out a unique divisor in $\Pic^h(\calY_g)$. This contributes $11$ to the rank of $\Pic^h(\calY_g)$.

To finish computing the rank of  $\Pic^h(\calY_g)$, we thus have to compute how many distinct divisors in $\calY_g$ are swept out by the $F_i$'s. However, as the divisors $F_i$ occur as the blow-ups of the singularities arising from the fixed points of the fibrewise Nikulin involution on $\calX_g$, each distinct divisor swept out by the $F_i$'s corresponds to an irreducible component of $C_g$. There are $s$ such components, so the rank of $\Pic^h(\calY_g)$ is equal to $11 + s$.

Next we consider the vertical divisors. As before, the class of a generic fibre contributes one divisor class to $\Pic^v(\calY_g)$; the remaining divisor classes arise from singular fibres. However, we computed the singular fibres in $\calY_g$ explicitly in the proof of Proposition \ref{KYgtrivial}. There we found that the fibre over a point $p$ with $g(p)=0$ and ramification order $x$ at $p$ has $2 +4(x-1) + (x - 2)^2 = x^2 + 2$ components (if $x$ is even) or $2 +4(x-1) + (x - 2)^2 -1 = x^2 + 1$ components (if $x$ is odd), so each fibre of this kind contributes $x^2 + 1$ (resp. $x^2$) new classes to $\Pic^v(\calY_g)$ when $x$ is even (resp. odd).

Furthermore, the fibre of $\calY_g$ over a point $p$ with $g(p)=\infty$ and ramification order $y$ at $p$ has $20$ (resp. $9$, $1$) components when $y = 1$ (resp. $2$, $4$). Thus, such fibres contribute $19$ (resp.  $8$, $0$) new classes when $y=1$ (resp. $2$, $4$). Summing over all singular fibres of $\calY_g$, we find that
\[\rank(\Pic^v(\calY_g)) = 1 + \sum_{x_i\ \mathrm{odd}} x_i^2 + \sum_{x_i\ \mathrm{even}} (x_i^2 + 1) + c_1 + c_2,\]
where $x_i$ and $c_j$ are as in the statement of the proposition. Adding in the $11 + s$ horizontal divisor classes, we obtain the result.
\end{proof}

\begin{proposition} \label{Ygh21} Let $\calY_g$ be a Calabi-Yau threefold as in Proposition \ref{YgCY3} and suppose that $g^{-1}(\infty)$ consists of two points \textup{(}so that $l=2$\textup{)}. Then 
\[h^{2,1}(\calY_g) =  k + \left(\dfrac{m_\mathrm{odd} - n}{2}\right) + p_g(C_g),\] 
where $k$ denotes the number of ramification points of $g$ over $\lambda = 0$, $m_{\mathrm{odd}}$ denotes the number of ramification points of odd order of $g$ over $\lambda = \frac{1}{256}$,  $n$ is the degree of $g$, and $p_g(C_g)$ denotes the geometric genus of the curve $C_g$ \textup{(}if $C_g$ is singular, this is equal to the sum of the genera of the components in the normalization of $C_g$\textup{)}.

Moreover, if $g$ is unramified over $\lambda = \frac{1}{256}$, then 
\[h^{2,1}(\mathcal{Y}_g) = k + p_g(C_g) = r + p_g(C_g),\] 
where $r$ is the degree of ramification of the map $g$ away from $\lambda \in \{0,\frac{1}{256},\infty\}$.
\end{proposition}
\begin{proof} To compute $h^{2,1}(\calY_g)$, we first find the third Betti number $b_3(\calY_g)$, then use the fact that $\calY_g$ is a Calabi-Yau threefold, so that $b_3(\calY_g) = 2h^{2,1}(\calY_g) + 2$.

As above, let  $U := g^{-1}(U_{M_2})$. Then let $j\colon U \to \Proj^1$ denote the inclusion map and let $\psi_U\colon \calY_g|_U \to U$ denote the restriction of the fibration $\psi_g\colon \calY_g \to \Proj^1$ to $U$. Applying \cite[Proposition 3.3]{cytfmqk3s}, noting that the condition on the singular fibres of $\calY_g$ is satisfied by the description of these fibres given in the proof of Proposition \ref{KYgtrivial}, we see that 
$H^3(\calY_g,\Q) \cong H^1(\Proj^1,j_*R^2(\psi_U)_*\Q)$.
It therefore suffices to compute the rank of this latter group. 

By the discussion in \cite[Section 2.1]{flpk3sm}, there is a splitting of $R^2(\psi_U)_*\Q$ as a direct sum of two irreducible $\Q$-local systems
\[R^2(\psi_U)_*\mathbb{Q} = \mathcal{NS}(\mathcal{Y}_g) \oplus \mathcal{T}(\mathcal{Y}_g),\]
where $\mathcal{NS}(\mathcal{Y}_g)$ consists of those classes which are in $\NS(Y_p) \otimes \mathbb{Q}$ for every smooth fibre $Y_p$ of $\mathcal{Y}_g$, and $\mathcal{T}(\mathcal{Y}_p)$ is the orthogonal complement of $\mathcal{NS}(\mathcal{Y}_g)$. We may therefore split
\[H^3(\calY_g,\Q) = H^1(\Proj^1,j_*R^2(\psi_U)_*\Q) = H^1(\Proj^1,j_*\mathcal{NS}(\mathcal{Y}_g)) \oplus H^1(\Proj^1,j_*\mathcal{T}(\mathcal{Y}_g)).\]

Now let $\pi_g\colon \calX_g \to \Proj^1$ denote the threefold fibred by $M_2$-polarized K3 surfaces related to $\psi_g\colon \calY_g \to \Proj^1$. Then, by \cite[Proposition 3.1]{flpk3sm}, the transcendental variations of Hodge structure $\mathcal{T}(\calY_g)$ and $\mathcal{T}(\calX_g)$ are isomorphic over $\R$. So, by \cite[Proposition 3.6]{cytfmqk3s}, we see that
\[h^1(\Proj^1,j_*\mathcal{T}(\calY_g)) = h^1(\Proj^1,j_*\mathcal{T}(\calX_g)) = 2 + 2k + (m_{\mathrm{odd}} - n).\]

Next we consider the $\Q$-local system $\mathcal{NS}(\mathcal{Y}_g)$. Let $\LL_{\mathrm{Nik}}$ denote the sub-$\Q$-local system of $\mathcal{NS}(\mathcal{Y}_g)$ generated by the classes of the divisors $\{F_1,\ldots,F_8\}$. Then it follows from \cite[Corollary 3.2]{flpk3sm} that there is a decomposition $\mathcal{NS}(\mathcal{Y}_g) \cong \Q^{12} \oplus \LL_{\mathrm{Nik}}$

We compute $\LL_{\mathrm{Nik}}$ as follows. Let $\pi_U\colon \calX_g|_U \to U$ (resp. $C_U$) denote the restriction of the fibration $\pi_g\colon \calX_g \to \Proj^1$ (resp. the curve $C_g \subset \calX_g$) to the open set $g^{-1}(U)$. Recall that the fixed locus of the fibrewise Nikulin involution on $\calX_g|_U$ is given by the curve $C_U$, and that the divisors $F_i \subset \calY_g$ arise from the resolution of the singularities in the quotient by this resolution. It therefore follows that there is an isomorphism of $\Q$-local systems
\[\LL_{\mathrm{Nik}} \cong (\pi_U|_{C_U})_*\Q_{C_U},\]
where $\Q_{C_U}$ denotes the constant sheaf with stalk $\Q$ on $C_U$.

Thus, we have 
\[h^1(\Proj^1,j_*\mathcal{NS}(\mathcal{Y}_g)) = h^1(\Proj^1,j_*\LL_{\mathrm{Nik}}) = h^1(\Proj^1,j_*(\pi_U|_{C_U})_*\Q_{C_U}).\]
Moreover, by the Leray spectral sequence, we have an isomorphism
\[H^i(\Proj^1,j_*(\pi_U|_{C_U})_*\Q_{C_U}) \cong H^i(\hat{C}_g,\Q),\] 
for every $i \geq 0$, where $\hat{C}_g$ denotes the normalization of $C_g$. We therefore find that
\[h^1(\Proj^1,j_*\mathcal{NS}(\mathcal{Y}_g)) = h^1(\hat{C}_g,\Q) = 2p_g(C_g).\]

Putting everything together, we obtain
\[b_3(\calY_g) = 2 + 2k + (m_{\mathrm{odd}} - n) + 2p_g(C_g),\]
and the result follows.
\end{proof}

\begin{remark} We note that the analogue of \cite[Proposition 4.1]{cytfmqk3s} is not true for $\calY_g$: even if we assume that $g$ is unramified over $\lambda = \frac{1}{256}$, in general we cannot realize every small deformation of $\calY_g$ by simply deforming the map $g$.

The reason for this is as follows. In \cite{cytfmqk3s}, the N\'{e}ron-Severi group of a general fibre in the K3 fibration on $\calX_g$ is isometric to $M_2$ and the smooth fibres form an $M_2$-polarized family. Consequently, the N\'{e}ron-Severi group of a general fibre is preserved under monodromy. It follows that any small deformation of $\calX_g$ is also fibred by $M_2$-polarized K3 surfaces, so can be realized as a pull-back of the family $\calX_2$, as seen in \cite[Proposition 4.1]{cytfmqk3s}.

However, in the setting presented here, we have seen that monodromy acts non-trivially on the N\'{e}ron-Severi group of a general fibre of $\calY_g$. Due to this, deformations need not preserve the entire N\'{e}ron-Severi group; they only need to preserve the part that is fixed under monodromy. In other words, any deformation of $\calY_g$ must be fibred by K3 surfaces, but the rank of the N\'{e}ron-Severi group of a general fibre in such fibrations may drop. Deformations of this type can obviously no longer be pull-backs of $\calY_2$.
\end{remark}

\begin{example} As a very simple example, we compute the Hodge numbers of the Calabi-Yau threefold $\calY_2'$. In this case $g$ is the map with $(k,l,m,n,r) = (4,2,4,8,0)$ and $[x_1,x_2,x_3,x_4] =[2,2,2,2]$, $[y_1,y_2] = [4,4]$, and $[z_1,z_2,z_3,z_4] = [2,2,2,2]$.

The action of monodromy around singular fibres of $\calY_2'$ fixes the N\'{e}ron-Severi lattice of a general fibre (by construction), so in particular it fixes the eight divisors $F_i$. The curve $C_g$ thus has eight components, all of which are rational curves. From Propositions \ref{Ygh11} and \ref{Ygh21} we  obtain 
\begin{align*} h^{1,1}(\calY'_2) & = 12 +  \sum_{x_i\ \mathrm{odd}} x_i^2 + \sum_{x_i\ \mathrm{even}} (x_i^2 + 1) + s + c_1 + c_2 \\
&= 12 + 0 + 4(5) + 8 + 0 + 0 \\
&= 40
\end{align*}
and $h^{2,1}(\calY_2') = k + \frac{1}{2}(m_\mathrm{odd} - n) + p_g(C_g) = 4 + \frac{1}{2}(0 - 8) + 0 = 0$, as expected from Proposition \ref{Y2'hodge}.
\end{example}

\begin{example}  As a harder example, we consider the map $g\colon \Proj^1 \to \calM_{M_2}$ defined by  $(k,l,m,n,r) = (1,2,5,5,1)$, $[x_1] = [5]$, $[y_1,y_2] = [1,4]$, and $[z_1,\ldots,z_5] = [1,1,1,1,1]$. By \cite[Theorem 5.10]{flpk3sm}, the Shioda-Inose fibred threefold $\calX_g$ corresponding to this map $g$ is birational to the mirror to the quintic threefold. We will compute the Hodge numbers of the corresponding Kummer fibred threefold $\calY_g$, which is Calabi-Yau by Proposition \ref{YgCY3}.

To do this, we begin by computing the curve $C_g$. This is given by the pull-back of $C_2$ by $g$. It has three irreducible components, given by the pull-backs of the irreducible components of $C_2$. Let $a_1 \in \Proj^1$ denote the unique point lying over $\lambda = 0$, $b_1 \in \Proj^1$ (resp. $b_2 \in \Proj^1$) denote the point over $\lambda = \infty$ where the ramification order of $g$ is $1$ (resp. $4$), and $d_1,\ldots,d_5$ denote the five points over $\lambda = \frac{1}{256}$. 

Then two of the irreducible components of $C_g$ (the pull-backs of the components of $C_2$ that are double covers of $\calM_{M_2}$) are double covers of $\Proj^1$ having a singularity of type $A_4$ at $a_1$, a simple ramification over $b_1$, and a singularity of type $A_3$ over $b_2$. The normalizations of these components are simply ramified over $a_1$ and $b_1$, so both have genus $0$.

The remaining irreducible component of $C_g$ (the pull-back of the component of $C_2$ that is a $4$-fold cover of $\calM_{M_2}$) is a $4$-fold cover of $\Proj^1$ having a pair of singularities of type $A_4$ over $a_1$, a $4$-fold ramification over $b_1$, a simple quadruple point over $b_2$, and ramification profile $[2,1,1]$ over each $d_i$. Its normalization has ramification profiles $[2,2]$ over $a_1$, $[4]$ over $b_1$, $[1,1,1,1]$ over $b_2$, and $[2,1,1]$ over each $d_i$. By Hurwitz's theorem, this normalization has genus $2$.

From this we can calculate the Hodge numbers of $\calY_g$. From Propositions \ref{Ygh11} and \ref{Ygh21} we  obtain 
\begin{align*} h^{1,1}(\calY_g) & = 12 +  \sum_{x_i\ \mathrm{odd}} x_i^2 + \sum_{x_i\ \mathrm{even}} (x_i^2 + 1) + s + c_1 + c_2 \\
&= 12 + 25 + 0 + 3 + 0 + 19 \\
&= 59
\end{align*}
and $h^{2,1}(\calY_g) = r + p_g(C_g) = 1 + 2 = 3$.
\end{example}

\bibliography{Books}
\bibliographystyle{amsplain}
\end{document}